\documentclass[11pt,twoside]{article}
\usepackage{latexsym}
\usepackage{amsthm,amssymb,amsbsy,amsmath,amsfonts,amssymb,amscd}
\usepackage{graphicx,color}
\usepackage{float,url}
\usepackage{cancel}
\usepackage[colorlinks,linkcolor=red,citecolor=blue]{hyperref}

\usepackage{mathtools} %dcases

\mathtoolsset{showonlyrefs}

\newcommand  \ind[1]  {   {1\hspace{-1.2mm}{\rm I}}_{\{#1\} }    }

\newcommand{\commentout}[1]{}

\newcommand{\R}{\mathbb{R}}

\newcommand {\al} {\alpha}
\newcommand {\eps}  {\varepsilon}

\newcommand {\p}   {\partial}
\newcommand{\beq}{\begin{equation}}
\newcommand{\eeq}{\end{equation}}
\newtheorem{theorem}{Theorem}[section]
\newtheorem{lemma}[theorem]{Lemma}

\newtheorem{remark}[theorem]{Remark}
\newtheorem{proposition}[theorem]{Proposition}
\newtheorem{corollary}[theorem]{Corollary}

\numberwithin{equation}{section}

\title{Noisy integrate-and-fire equation: continuation after blow-up}

\author{Xu'an Dou\thanks{Beijing International Center for Mathematical Research, Peking University, Beijing, 100871, China. Email: dxa@pku.edu.cn}
\and
Beno\^ \i t Perthame\thanks{Sorbonne Universit{\'e}, CNRS, Universit\'{e} de Paris, Inria, Laboratoire Jacques-Louis Lions, F-75005 Paris. 
Email : Benoit.Perthame@sorbonne-universite.fr}
\and 
Delphine Salort\thanks{Sorbonne Universit{\'e}, CNRS, Laboratoire de Biologie Computationnelle et Quantitative, F-75005 Paris. Email : delphine.salort@sorbonne-universite.fr} \thanks{Supported by ANR-19-CE40-0024.}
\and 
Zhennan Zhou\thanks{Institute for Theoretical Sciences, Westlake University, Hangzhou, Zhejiang Province, 310030, China. Email: zhouzhennan@westlake.edu.cn}
\thanks{Supported by the National Key R\&D Program of China, Project Number 2021YFA1001200, and the NSFC, grant Number 12031013, 12171013.}
}

\date{\today}

\begin{document}

\maketitle

\pagestyle{plain}
%\tableofcontents
\pagenumbering{arabic}

\begin{abstract} 
The integrate and fire equation is a classical model for neural assemblies which can exhibit finite time blow-up. A major open problem is to understand how to continue solutions after blow-up.

Here we study an approach based on random discharge models and a change of time which generates a classical global solution to the expense of a strong absorption rate $1/\eps$. We prove that in the limit $\eps\to0^+$, a global solution is recovered where the integrate and fire equation is reformulated with a singular measure. This describes the dynamics after blow-up and also gives information on the blow-up phenomena itself.

The major difficulty is to handle nonlinear terms. To circumvent it, we establish two new estimates, a kind of equi-integrability of the discharge measure and a $L^2$ estimate of the density. The use of the new timescale turns out to be fundamental for those estimates.
\end{abstract} 
\vskip .7cm

\noindent{\makebox[1in]\hrulefill}\newline

% 

%35K55 Nonlinear parabolic equations 

%

%35Q84(Fokker-Planck equations);
%35Q92(PDEs in connection with biology, chemistry and other natural sciences
%35B44  	Blow-up in context of PDEs
% 35Dxx		Generalized solutions to partial differential equations
%   92B20(Mathematical biology in general-Neural networks for/in biological studies, artificial life and related topics)

%\newline


%------------------------
\section{Introduction}
\label{sec:intro}
The nonlinear noisy leaky integrate-and-fire equation (NNLIF in short) has been introduced to represent some homogeneous neural networks, see \cite{BrHa, brunel} and the references therein.  In this model, each neuron is governed by the integrate-and-fire dynamics with noise, and each neuron receives instantaneously the mean activity of the network. It is now a well-established continuous description derived from random finite size networks,  \cite{delarue2015particle,BossyFT,FL2016,jabin2023meanfield}. The NNLIF equation describes the probability $p(t,v)$ to find a neuron with membrane potential $v$, which takes the form of the drift-diffusion equation
\begin{align}
   & \frac{\partial p}{\partial t} +\frac{\partial }{\partial v} \big[\big(-v+ bN(t) \big)p\big]- a \frac{\partial^2 p}{\partial v^2} = N(t) \delta_{V_R}(v), \qquad t \geq 0, \, v \leq V_F, \label{eq:nif}
\\
  & p(t, V_F)=0, \qquad N(t):= -a \frac{\partial p(t,V_F) }{\partial v}. \label{eq:bry}
\end{align}
Here the reset and firing potentials are given numbers satisfying $V_R < V_F$. The parameter $b$ represents the network connectivity and gives rise to the mathematical interest of the equation since it generates a quadratic non-linearity with two difficulties: the non-linearity arises from the boundary flux according to~\eqref{eq:bry} and acts on the drift. When  $b \leq 0$ (inhibitory or non-connected network), Eq.~\eqref{eq:nif} admits global bounded solutions \cite{CGGS}. This is also the case when $b>0$ is small enough and if the initial data is ``well-behaved'', see \cite{carrillo2014qualitative,carrillo2024classical, CiCP-30-820, RouxSalort}, see also \cite{DelarueIRT2015AAP} for a probability viewpoint and \cite{caceres2024,caceres2024sequence} for recent progress on the long time behavior.

However, when $b>0$, solutions may blow up in finite time, \cite{CCP, delarue2015particle}, where the blow-up time may represent the network synchronisation~\cite{Henry:13}, and is connected to  the multiple firing event in computational neuroscience \cite{zhang2014distribution}. This situation is more intriguing and the question of understanding what happens after the blow-up had a growing interest in recent years. Theoretical construction of the solution after blow-up has been achieved, with a probabilistic viewpoint, in \cite{delarue2015particle}. It has generated broader interests, mainly with probability approaches, as blow-ups of similar nature also arise in models beyond neuroscience, including financial models \cite{hambly2019mckean,2019-AAP1403} and the supercooled Stefan problem \cite{delarue2022global} (see also \cite{CGGS,LS_2020} for the connection between NNLIF and Stefan equations). Numerical investigations with the help of a finite neuron network have been carried out in \cite{CiCP-30-820,du2024synchronization}.
%\xd{not sure the word ``sparked'' in ``It has sparked broader interests''}

An analytical approach towards defining the solution after blow-up, based on PDE tools, has been proposed and studied, independently and differently, in~\cite{dou2022dilating} and \cite{taillefumier2022characterization,sadun2022global}. The key idea is to introduce a new timescale, which is related to the firing rate $N(t)$ and dilates the time near the blow-up when $N(t)=+\infty$. Both \cite{dou2022dilating} and \cite{taillefumier2022characterization,sadun2022global} treat the case when the diffusion coefficient depends linearly on the network activity. Indeed, this assumption allows to preserve a uniform lower bound for the diffusion coefficient, even after dilating time. Here we extend this approach when the diffusion coefficient is constant; the main difficulty is that, after time dilation, the equation degenerates near the blow-up time, and so we have to derive a new analysis in order to overcome this new difficulty. We also refer to  \cite{carrillo2024classical} where the new timescale becomes the key for analyzing the long time behavior.

To introduce our approach, we first rewrite \eqref{eq:nif} before blow-up as an equation on the whole line as in \cite{ikeda2022theoretical} % (see also calculations in Section \ref{subsec:inter-cl}) 
\begin{equation} \label{eq:nif-2delta} 
\begin{cases}
        \frac{\partial p}{\partial t} +\frac{\partial }{\partial v} \big[\big(-v+ bN(t) \big)p\big]- a \frac{\partial^2 p}{\partial v^2} = N(t) \delta_{V_R}(v)-N(t) \delta_{V_F}(v), \qquad t \geq 0, \, v \in \R, 
        \\[10pt]
        p(t, v)=0 \quad \text{for } v>V_F, \qquad N(t):=-a \frac{\partial p (t, V_F^-)}{\partial v}.
\end{cases}\end{equation}Following \cite{taillefumier2022characterization,sadun2022global,dou2022dilating}, we introduce a change of time
\begin{equation}
    N(t) dt =d \tau, \qquad n(\tau, v)=p(t,v), \quad \text{and}\quad  Q(\tau)=\frac{1}{N(t)}.\label{def:tau}
\end{equation} The new timescale $\tau$ is called the dilated timescale in \cite{dou2022dilating} as it dilates the time when $N(t)$ approaches infinity.
A simple use of the chain rule shows that the NNLIF equation becomes
\begin{equation}\label{eq:Qnif-cl}
    \frac{\partial n}{\partial \tau} +\frac{\partial }{ \partial v} \big[\big(-v Q(\tau)+ b \big)n \big]- a Q(\tau) \frac{\partial^2 n}{\partial v^2} = \delta_{V_R}(v)-\delta_{V_F}(v).
\end{equation}
However, in order to take into the boundary condition after blow-up, our construction gives a triple $(n, Q, {\cal S})$ such that %and based on the time dilation regularisation method,
\begin{align}
 &    \frac{\partial n}{\partial \tau} +\frac{\partial }{ \partial v} \big[\big(-v Q(\tau)+ b \big)n \big]- a Q(\tau) \frac{\partial^2 n}{\partial v^2} = \delta_{V_R}(v)- {\cal S}(\tau,v) , \qquad \tau \geq 0, \, v \in \R. \label{eq:Qnif}
\end{align}
The probability measure ${\cal S}(v)$ replaces the Dirichlet boundary condition at $V_F$ just as $\delta_{V_F}(v)$ does in \eqref{eq:Qnif-cl} before blow-up. To include continuation after blow-up,  since, in the degenerate case when the Dirichlet condition cannot be imposed, it will satisfy % and thus
\begin{align}
\begin{cases}  
{\cal S}(\tau, v) \geq 0, \qquad \int_{-\infty}^{+\infty} {\cal S}(\tau, v) dv=1,  \qquad {\cal S}(\tau, v) =0 \quad \text{for} \quad  v<V_F, \label{eq:Qbry}
\\[10pt]
 {\tau \mapsto  \int_{V_F}^\infty n(\tau, v)dv \quad \text{ is continuous},}
\\[10pt]
\text{when }  \int_{V_F}^\infty n(\tau, v)dv >0, \quad \text{then} \quad Q(\tau)=0,
 \\[10pt]
\text{when }  \int_{V_F}^\infty n(\tau, v)dv =0, \quad \text{then} \quad\quad {\cal S}(\tau, v) = \delta_{V_F}(v),\\[10pt]
\text{when $Q(\tau)>0$}, \quad \text{then } \int_{V_F}^\infty n(\tau, v)dv=n(\tau, V_F)=0 \text{ and } -a\p_vn(\tau,V_F)Q(\tau) =1.
\end{cases} 
\end{align}

The blow-up times, i.e. $N(t)=\infty$, correspond to  $Q(\tau)=0$ and the change of variable \eqref{def:tau} is singular. Then Eq.~\eqref{eq:Qnif} degenerates, propagating information beyond $V_F$ since it is reduced to
\begin{align}\label{eq:Qnif-bl}
 \frac{\partial n}{\partial \tau} + b \frac{\partial n}{\partial v} = \delta_{V_R}(v)- {\cal S}(\tau,v), \quad \text{in} \quad \{\tau \in (0,\infty)\, \int_{V_F}^\infty n(\tau, v)dv >0\}.
\end{align}
In any interval where $Q(\tau)>0$, one can check that the conclusion  ${\cal S}(\tau, v) = \delta_{V_F}(v)$ means that the solution corresponds, in the variable $t$, to a solution of the NNLIF problem \eqref{eq:nif}--\eqref{eq:bry}.  Indeed, $(Q,{\cal S})$ might be viewed as a Lagrangian multiplier to keep $n(\tau,\cdot)$ as a probability measure, regardless of whether or not a blow-up is occurring.

We use the notation ${\cal M}$ for the space of bounded measures and ${\cal M}_+\subset {\cal M}$ for the non-negative ones. Our purpose is to prove the 
%------------------
\begin{theorem} \label{th:main}
Let the  initial data $n^0$ satisfy
\begin{align} \label{as:ID}
n^0\geq 0, \quad \int_\R n^0=1, \quad \int_\R v^2n^0 < \infty,\quad \int_\R (n^0)^2 <\infty.
\end{align}
Then, there is a global weak solution of \eqref{eq:Qnif}--\eqref{eq:Qbry}  with $n \in L^\infty\big((0, \infty); L^1(\R) \big)$, and for all $\tau_0>0$, in each interval $(0,\tau_0)$, $Q(\cdot)\in {\cal M}_+(0, \tau_0)$, $n\in L^\infty\big((0,\tau_0); L^2(\R) \big)$ and $\int_\R v^2n(\tau,v) \leq C(\tau_0)$. 
Furthermore, for all $\psi\in L^2+C_b (\R)$,  $\int_\R \psi(v) n(\tau,v)dv \in C(0,\infty)$ and the complementary relation holds
\begin{equation}\label{eq:complementary-relation}
Q(\tau)\int_{V_F}^{+\infty}n(\tau,v)dv=0, \quad \forall \tau >0.
\end{equation}
\end{theorem}

{ This theorem not only establishes the global existence of the solution in $\tau$ timescale, but it also has several consequences on the continuation after the blow-up in $t$ timescale (see Theorem \ref{th:lifespan}): the blow-up size and the post-blow-up profile,  as well as the lifespan in $t$ timescale. In particular,  global solution is obtained as soon as $b< V_F-V_R$ while, when $b> V_F-V_R$, the time $t$ lifespan can be finite. }

Theorem \ref{th:main} follows from Theorems \ref{thm:uni-in-eps-bounds} and \ref {thm:limit} below. Our proof strategy differs deeply from the previous works \cite{delarue2015global,delarue2022global,dou2022dilating,taillefumier2022characterization}: we use a random discharge equation set on the full line \cite{CP2014, JGLiuZZ21} which is described in Section~\ref{sec:reg}. This involves a strong absorption term of strength $\eps^{-1}$ and generates an approximate solution $(n_\eps, Q_\eps,{\cal S}_{\eps})$. The question is then to find appropriate estimates to pass to the limit $\eps\to0^+$, for which the use of the new timescale turns out to be essential. Two key estimates, a uniform estimate on $Q_{\eps}$ and a uniform $L^2$ estimate on $n_{\eps}$ are established in Section~\ref{sec:equi} and \ref{sec:L2}. Those estimates allow us to pass to the limit $\eps \to 0^+$ and characterize the limit solution in Section~\ref{sec:limit}. More properties of  the solution  are given in Section~\ref{sec:6}. Finally, conclusions and discussion are given in Section~\ref{sec:7}.

%------------------------------------------
\section{Regularized problem and moment estimates}
\label{sec:reg}
%-----------------------------------------

 The blow-up phenomena is related to the firing mechanism in \eqref{eq:nif}-\eqref{eq:bry}: neurons fire as soon as their voltages reach the fixed threshold $V_F$, as reflected in  the boundary flux definition of $N(t)$~\eqref{eq:bry}. As a regularized problem, we consider a model with a random firing mechanism, called the random discharge (or firing) model, see  \cite{CP2014,JGLiuZZ21}.   In this regularized problem, the deterministic firing is relaxed to a random firing with its firing intensity proportional to $\frac{1}{\eps}$ for $v>V_F$. The firing rate is defined as an integral which is bounded a priori by $\frac{1}{\eps}$, thus ensuring a global solution. 
 
 Our goal is to prove the $\eps\rightarrow0^+$ limit which could/might give a solution beyond blow-ups for the original problem. A key perspective here is to work in the dilated timescale $\tau$ as defined in~\eqref{def:tau}, which allows us to establish various estimates uniformly in $\eps$.

%------------------------------------------
\subsection{Random discharge in dilated timescale} 
%------------------------------------------

For $\eps>0$ given, the random discharge model in the original timescale $t$ reads
\begin{equation} \label{eq:nif-random}
\begin{cases}
    %&\frac{\partial p_{\eps}}{\partial t} +\frac{\partial }{\partial v} \big[\big(-v+ bN_{\eps}(t) \big)p_{\eps}\big]- a \frac{\partial^2 p_{\eps}}{\partial v^2} = N_{\eps}(t) \delta_{V_R}(v)-\phi_{\eps}(v)p_{\eps}(t,v), \qquad t \geq 0, \, v \in\mathbb{R}, 
&\p_{t} p_{\eps} +\p_v [\big(-v+ bN_{\eps}(t) \big)p_{\eps}]- a {\partial_{vv} p_{\eps}} = N_{\eps}(t) \delta_{V_R}(v)-\phi_{\eps}(v)p_{\eps}(t,v), \qquad t \geq 0, \, v \in\mathbb{R}, 
\\[10pt]
   & N_{\eps}(t):=\int_{V_F}^{+\infty}\phi_{\eps}(v)p_{\eps}(t,v),
\end{cases} \end{equation}
with the absorption profile (only chosen for its simplicity) defined as
\begin{equation} \label{eq:phi}
    \phi_{\eps}(v)=\frac{1}{\eps}\phi(v)=\frac{1}{\eps}\mathbb{I}_{v\geq V_F},
\end{equation} 
which gives a priori in \eqref{eq:nif-random}
\begin{equation}\label{bd-N-eps}
     N_{\eps}\leq \frac{1}{\eps}.
\end{equation} 
\begin{remark}\label{rmk:uplwN}
    For $\eps>0$, $N_{\eps}$ is also bounded from below. A positive lower bound for $N_{\eps}$ cannot be established pointwise, as it can vanish initially if the initial data is supported on $(-\infty,V_F]$. However, we have integral-in-time positive lower bounds, see Prop.~\ref{prop:t-lower-N} in Section~\ref{sec:equi}. 
\end{remark}

The bound \eqref{bd-N-eps} ensures that the solution is global in time for every fixed $\eps>0$. Intuitively, when $\eps\rightarrow0^+$, this random discharge model converges to the model with the fixed threshold $V_F$, see e.g.~\cite{JGLiuZZ21} for a justification in the linear case when $b=0$. However, it is difficult to prove uniform-in-$\eps$ bounds when $b>0$, as the limit $N(t)$ can blow up. Here, in contrast to previous literature \cite{CP2014,JGLiuZZ21} in timescale~$t$, we consider the regularized problem in the dilated timescale~$\tau$ defined in~\eqref{def:tau}, 
\beq \begin{cases}
 \p_\tau n_\eps(\tau,v) + \p_v [(-v Q_\eps(\tau) + b) n_\eps] - a Q_\eps(\tau) \p_{vv} n _\eps = \delta_{V_R}(v)-Q_\eps(\tau)  \phi_\eps n_\eps , \quad \tau \geq 0 , \; v\in \R,
 \\[5pt]
 Q_\eps(\tau) = \Big(\int_\R \phi_\eps(v) n_\eps (\tau,v) dv\Big)^{-1},  
 \\[5pt]
 n_{\eps}(\tau=0,v)=n^0(v),\qquad v\in\R,%\int_\R n_\eps (\tau,v) dv =  \int_\R n^0 (v) dv = 1,
 \end{cases}
\label{IFnew}
\eeq 
where our choice of $\phi_{\eps}$ in \eqref{eq:phi} and \eqref{bd-N-eps} give a priori
\begin{equation}\label{lower-bound-Q}
     Q_{\eps}\geq \eps.
\end{equation}

For a fixed $\eps>0$, the change of time \eqref{def:tau} is regular, thanks to the upper bound and the integrate-in-time positive lower bound of $N_{\eps}(t)$, see Remark~\ref{rmk:uplwN}. Therefore \eqref{eq:nif-random} and \eqref{IFnew} are effectively equivalent, when $\eps>0$ is fixed. However, it is in the new timescale $\tau$ that we can prove various uniform-in-$\eps$ bounds and establish the global-in-time limit as $\eps\rightarrow0^+$, even across intervals where $Q(\tau)=0$, corresponding to blow-up for $N(t)$.

We define the loss term in the right hand side of \eqref{IFnew} as
\begin{equation}\label{def-Seps}
        {\cal S}_{\eps}(\tau,v):=Q_{\eps}(\tau)\phi_{\eps}(v)n_{\eps}(\tau,v)=\frac{1}{\eps}Q_{\eps}(\tau)n_{\eps}(\tau,v)\mathbb{I}_{v\geq V_F}=\frac{n_{\eps}(\tau,v)\mathbb{I}_{v\geq V_F}}{\int_{V_F}^{+\infty}n_{\eps}(\tau,v)dv},
\end{equation} which we shall show converges to the measure ${\cal S}$ in Theorem~\ref{th:main}. The following properties of ${\cal S}_{\eps}$ are directly derived from its definition
\begin{equation}
\int_{\mathbb{R}}{\cal S}_{\eps}(\tau,v)dv=1, \qquad  \text{supp} \;{\cal S}_{\eps}(\tau,\cdot) \subseteq [V_F,+\infty).
\end{equation}

\subsection{Main results on the regularized problem}

In the timescale $\tau$, we are able to establish the limit $\eps\rightarrow0^+$ and in this way construct a global solution of \eqref{eq:Qnif}, thus proving Theorem~\ref{th:main}. We make this precise in the two following theorems
%-------------------------------------------------
\begin{theorem}[Uniform-in-$\eps$ bounds]\label{thm:uni-in-eps-bounds} 
Assume \eqref{as:ID} on initial data. The solution $(n_\eps,Q_{\eps},S_{\eps})$ of~\eqref{IFnew} satisfies the following bounds with uniform-in-$\eps$ constants, on every finite time interval $(0,\tau_0)$,
\[
 \int_{\R} v^2n_{\eps}(\tau,v)dv\leq C(\tau_0), \qquad \int_{\R}n_{\eps}^{2}(\tau,v)dv\leq C(\tau_0),
\]
\[
\tau \mapsto  \int_{\R} \psi(v) n_{\eps}(\tau,v)dv \quad  \text{is continuous uniformly in $\eps$  for all} \quad \psi \in L^2+C_b,
\]
\[
    \int_0^{\tau_0}\int_{\mathbb{R}}(v-V_F)_+^2S_{\eps}(\tau,v)d\tau dv \leq C(\tau_0),
\]
\[
\int_0^{\tau_0}Q_{\eps}(\tau)d\tau\leq C(\tau_0), \qquad 
\int_{\tau}^{\tau+\delta}Q_{\eps}(s)ds \leq \frac{C(\tau_0)}{|-\ln \delta |},
\]
 which holds for $0<\eps<1$, $0\leq \tau\leq \tau_0$ and $0<\delta\leq\frac{1}{2}$.
\end{theorem} 

In the above bounds, the first line gives the second moment and $L^2$ estimates of $n_{\eps}$ while the second line states a weak continuity in time of $n_{\eps}$, where $\psi$ is allowed to be in the sum space $L^2+C_b$. The third line gives a tightness bound for ${\cal S}_{\eps}$. In the fourth line, two $L^1$ estimates on $Q_{\eps}$ is given: one on~$[0,\tau_0]$ and a refined one on a small interval $[\tau,\tau+\delta]$. 

Based on these estimates, we can derive the
%---------------------
\begin{theorem}[The $\eps\rightarrow0^+$ limit]\label{thm:limit}
Assume \eqref{as:ID} on initial data. As $\eps\rightarrow0^+$, after extracting subsequences, the solution of \eqref{IFnew} $(n_\eps,Q_{\eps},S_{\eps})$ converges to $(n,Q,S)$ in the weak sense following the estimates of Theorem~\ref{thm:uni-in-eps-bounds}.  Furthermore, $(n,Q,S)$ is a global-in-time weak solution of the limit system \eqref{eq:Qnif}, and \eqref{eq:Qbry} holds as stated in Prop.~\ref{Prop:S-Dirac} and Prop.~\ref{prop:limit-Ibl}.
\end{theorem}

\begin{remark}
We stress that the limit solution is global in time, which reflects an advantage working in~$\tau$-timescale. For $\eps>0$, both the $t$-system \eqref{eq:nif-random} and the $\tau$-system \eqref{IFnew} have global solutions. However, this might not be the case in the limit, and we need additional efforts to recover the obtained global solution in $t$ from the global solution in $\tau$, see Section~\ref{sec:back-t}. 

\end{remark}

The proof of Theorem \ref{thm:limit} is given in Section \ref{sec:limit}, where we also specify the sense of the convergence. It relies on the uniform-in-$\eps$ estimates summarized in Theorem \ref{thm:uni-in-eps-bounds}. 

For Theorem \ref{thm:uni-in-eps-bounds}, we divide and prove it in four propositions: The estimates using moments are given in Prop.~\ref{prop:second-moment} below. The more intriguing ones: the refined uniform bound for $Q_\eps$ and the $L^2$ estimate for $n_{\eps}$ are given in Prop.~\ref{prop:finer-Q} in Section~\ref{sec:equi} and Prop.~\ref{prop:redL2} in Section~\ref{sec:L2}, respectively. Then in Prop.~\ref{prop:equi-continuity}, we prove the weak continuity in time for $n_{\eps}$. 

Before starting the proof, we note that for fixed $\eps>0$ it is standard to show the existence of a weak solution of \eqref{IFnew}, satisfying that for each test function $\psi(v)\in C^2_b(\mathbb{R})$,
\begin{equation}\label{weak-test}
    \begin{aligned}
        \frac{d}{d\tau}\int_{\mathbb{R}}\psi(v)n_{\eps}(\tau,v)dv=&\int_{\mathbb{R}}b\p_{v}\psi(v)n_{\eps}(\tau,v)dv +\psi(V_R)\\&+ Q_{\eps}(\tau)\int_{\mathbb{R}}\bigl(a\p_{vv}\psi-v\p_v\psi-\phi_{\eps}(v)\psi(v)\bigr)n_{\eps}(\tau,v)dv.
    \end{aligned}
\end{equation}

\subsection{Second moment and related controls}
%----------------------------------------------

Our first goal is to prove the simplest estimates useful for the sequel, which we recall here

%---------------------------------
\begin{proposition}[Second moment of $n_\eps$ and integrability of $Q_{\eps}$]\label{prop:second-moment}
Assuming \eqref{as:ID}, the solution of \eqref{IFnew} satisfies the following bounds with constants independent of $\eps$, for all $\tau_0>0$ and $0\leq \tau \leq \tau_0$.
\begin{align} 
\label{eq:secondmoment}
&\int_{\mathbb{R}} v^2 n_\eps(\tau, v) dv \leq C(\tau_0),\\
\label{eq:est_Qtau}
    &\int_0^{\tau_0}Q_{\eps}(\tau)d\tau \leq C(\tau_0),\\
    \label{est_m2rbis}
&\int_0^{\tau_0} Q_{\eps}(\tau)\int_{\mathbb{R}}(v-V_F)_+^2n_{\eps}(\tau,v)dv d\tau \leq C(\tau_0) \eps.
\end{align}  
\end{proposition} 
\begin{remark}\label{tightness-Seps}
    We note that \eqref{est_m2rbis} can be reformulated as a tightness bound for ${\cal S}_{\eps}$ defined in \eqref{def-Seps}
    \begin{equation}
        \int_0^{\tau_0}\int_{\mathbb{R}}(v-V_F)_+^2S_{\eps}(\tau,v)d\tau dv \leq C(\tau_0).
    \end{equation}
\end{remark}

The rest of this section is devoted to the proof of Prop. \ref{prop:second-moment}.

\paragraph{Moment control for $v>V_F$.}

Taking $\psi(v)=\frac{1}{2}(v-V_F)_+^2$ (integrate $\phi(v)$ twice), then we have $\psi''(v)=\phi(v)$ and
\begin{align*}
    \frac{d}{d\tau}\int_{\mathbb{R}}&\frac{1}{2}(v-V_F)_+^2n_{\eps}(\tau,v)dv=\int_{\mathbb{R}}b(v-V_F)_+n_{\eps}(\tau,v)dv+0\\&+ Q_{\eps}(\tau)\int_{V_F}^{+\infty}\bigl(a-v(v-V_F)_+)n_{\eps}(\tau,v)dv-Q_{\eps}(\tau)\int_{\mathbb{R}}\frac{1}{2\eps}(v-V_F)_+^2n_{\eps}(\tau,v)dv 
\end{align*}    
Note that by definition of $Q_{\eps}(\tau)$
\begin{equation} \label{Qninfini}
    Q_{\eps}(\tau)\int_{V_F}^{+\infty}an_{\eps}(\tau,v)dv=a\eps,
\end{equation} 
which gives 
\begin{align}
        \frac{d}{d\tau}\int_{\mathbb{R}}\frac{1}{2}(v-V_F)_+^2n_{\eps}(\tau,v)dv+Q_{\eps}(\tau)\int_{\mathbb{R}}\bigl((v-V_F)_+^2&+V_F(v-V_F)_+\bigr)n_{\eps}(\tau,v)dv \notag
        \\
        +Q_{\eps}(\tau)\int_{\mathbb{R}}\frac{1}{2\eps}(v-V_F)_+^2n_{\eps}(\tau,v)dv&=b\int_{\mathbb{R}}(v-V_F)_+n_{\eps}(\tau,v)dv+ a\eps  \notag
        \\ 
        &\leq  b\left(\int_{\mathbb{R}}(v-V_F)_+^2n_{\eps}(\tau,v)dv\right)^{1/2}+a\eps. \label{tmp-1}
\end{align} Using $(v-V_F)^2_++V_F(v-V_F)_+\geq -\frac{V_F^2}{4}$ for $v>V_F$ (regardless of the sign of $V_F$), we get
\begin{equation}
Q_{\eps}(\tau)\int_{\mathbb{R}}\bigl((v-V_F)_+^2+V_F(v-V_F)_+\bigr)n_{\eps}(\tau,v)dv \geq -\frac{V_F^2}{4} Q_{\eps}(\tau) \int_{V_F}^{+\infty} n_{\eps}(\tau,v)dv=-\frac{V_F^2}{4}\eps,
\end{equation} which together with \eqref{tmp-1} gives
\begin{align*}
        \frac{d}{d\tau}\int_{\mathbb{R}}\frac{1}{2}(v-V_F)_+^2n_{\eps}(\tau,v)dv
        &+Q_{\eps}(\tau)\int_{\mathbb{R}}\frac{1}{2\eps}(v-V_F)_+^2n_{\eps}(\tau,v)dv\\ &\leq  b\left(\int_{\mathbb{R}}(v-V_F)_+^2n_{\eps}(\tau,v)dv\right)^{1/2}+\left(a+\frac{V_F^2}{4}\right)\eps \\ &\leq  {b^2}\left(\int_{\mathbb{R}}(v-V_F)_+^2n_{\eps}(\tau,v)dv\right)+1+C\eps,
\end{align*} where in the last line we used $\sqrt{x}\leq x+1$. By the Gronwall lemma, we conclude that
\begin{equation}\label{est_m2r}
    \int_{\mathbb{R}}\frac{1}{2}(v-V_F)_+^2n_{\eps}(\tau_0,v)dv+ \int_0^{\tau_0} Q_{\eps}(\tau)\int_{\mathbb{R}}\frac{1}{2\eps}(v-V_F)_+^2n_{\eps}(\tau,v)dv d\tau \leq C(\tau_0),
\end{equation}
which proves the statement \eqref{est_m2rbis}.

\paragraph{Control of the second moment.} 
Taking $\psi(v)=v^2$ we obtain
\begin{equation}  \label{secondm1}
\frac{d}{d\tau}\int_{\mathbb{R}} v^2 n_\eps(\tau, v) dv = 2b \int_{\mathbb{R}} v n_\eps(\tau, v) dv +V_R^2+ Q_\eps(\tau)\int_{\mathbb{R}}  \big( 2a-2v^2 -\phi_{\eps}(v) v^2 \big) n_\eps(\tau, v) dv.
\end{equation}
From this we deduce 
\[
\frac{d}{d\tau}\int_{\mathbb{R}} v^2 n_\eps(\tau, v) dv \leq b^2+V_R^2 +\int_{\mathbb{R}} v^2 n_\eps(\tau, v) dv+ Q_\eps(\tau) \Big( 2a -2 \int_{\mathbb{R}}v^2  n_\eps(\tau, v) dv \Big),
\]
 which can be rewritten as
\begin{equation}
    \frac{d}{d\tau}\left(\int_{\mathbb{R}} v^2 n_\eps(\tau, v) dv-a\right) \leq C+(1- 2Q_\eps(\tau) )\left(  \int_{\mathbb{R}}v^2  n_\eps(\tau, v) dv -a\right),
\end{equation}
from which we deduce by the Gronwall lemma that
\begin{equation*}
\int_{\mathbb{R}} v^2 n_\eps(\tau, v) dv \leq C(\tau).
\end{equation*}
thus proving the first statement \eqref{eq:secondmoment} of Prop. \ref{prop:second-moment}.

Furthermore, integrating in time \eqref{secondm1} and using \eqref{eq:secondmoment} and \eqref{est_m2rbis}, we also get %thanks to the controls for $v>V_F$ that 
\begin{equation}  \label{secondm2} 
-C(\tau_0)\leq \int_0^{\tau_0}Q_\eps(\tau) \int_{\mathbb{R}}(v^2-a ) n_\eps(\tau, v) dv \leq C(\tau_0).
\end{equation} 
This is not enough to derive the integrability of $Q_{\eps}$ since $v^2-a$ can be negative  when $|v|\leq \sqrt a $. This motivates us to find a local estimate.

\paragraph{Local integrability of $Q_{\eps}(\tau)$.} 

To prove  integrability of $Q_{\eps}(\tau)$ we need to combine  the relation \eqref{secondm2} with a local version. To do so, we choose a test function  $\psi_A(v)\in C^1(\R)$, for $A$ large,  as follows
\[
a\p_{vv}\psi_A-v\p_v\psi_A =a \qquad \text{in } (-A,A]
\]
and $\psi_A(v)= 0$ for $v\leq -A$, and $\psi_A(v)$ is affine for $v\geq A$. The explicit formula for $\psi_A$ is deduced from its derivative that, for $v\in[-a, a]$, we choose as 
\[
\p_v \psi_A(v) =  e^{\frac{v^2}{2a }}\int_{-A}^v e^{-\frac{w^2}{2a }}dw>0 \qquad  \text{thus } \p_v \psi_A(-A)=0.
\]
In the weak formulation \eqref{weak-test}, all the the terms are under control (note that $\psi_A(v)\leq C(1+v_+)$ and $\psi_A',\psi_A''$ are bounded, and we can also choose $A>V_F$) from the previous estimates and we infer that
\[
\int_0^{\tau_0}Q_\eps(\tau)\int_{-A}^A n_\eps(\tau, v) dv \leq C(\tau_0).%=O(1).
\] 

We may now conclude, combining the above estimate and \eqref{secondm2}, that
\begin{align*}
 \int_0^{\tau_0} Q_\eps(\tau)d\tau & = \int_0^{\tau_0} Q_\eps(\tau) \int_{\mathbb{R}} n_\eps(\tau, v) dvd\tau 
 \leq \int_0^{\tau_0} Q_\eps(\tau) \left[ \frac 1{A^2}  \int_{\mathbb{R}}v^2  n_\eps(\tau, v) dv +\int_{-A}^A  n_\eps(\tau, v) dv \right] d\tau \notag
\\
& \leq \frac 1{A^2}\int_0^{\tau_0} Q_\eps(\tau) \left[   \int_{\mathbb{R}}(v^2-a)  n_\eps(\tau, v) dv +a\int_{\R}  n_\eps(\tau, v) dv \right] d\tau +C\notag
 \\
 & \leq \frac a {A^2}  \int_0^{\tau_0} Q_\eps(\tau)d\tau +C
\end{align*}
and, taking $A^2 >a$,  we conclude  \eqref{eq:est_Qtau}
\begin{equation}%\label{est_Qtau}
    \int_0^{\tau_0}Q_{\eps}(\tau)d\tau \leq C(\tau_0).
\end{equation}
This concludes the proof of  Prop. \ref{prop:second-moment}.
\qed

%------------------------------------------------------
\section{Refined uniform bounds for $Q_\eps$} 
\label{sec:equi}

The previous controls in Prop. \ref{prop:second-moment} are not enough to pass to the limit in the nonlinear terms of~\eqref{IFnew}. In particular, we need to improve the simple $L^1$ bound on $Q_{\eps}$ \eqref{eq:est_Qtau}. In the following, we turn to prove a refined uniform bound which shows that the integral of $Q_\eps$ is small on a small interval. This will imply the weak continuity in time of $n_{\eps}$ in Prop.~\ref{prop:equi-continuity}, which is crucial for passing the limit in nonlinear terms.

\subsection{Statement of the refined bound}

\begin{proposition}[{Refined} uniform bounds for $Q_\eps(\tau)$]\label{prop:finer-Q}
Assume \eqref{as:ID} and fix a $\tau_0>0$. For the solution of \eqref{IFnew}, there is a constant $C(\tau_0)$ such that for all  $0\leq \tau<\tau_0$, $0<\eps<1$ and $0<\delta \leq  \frac 1 2$, we have 
\begin{equation}\label{est_finer_Q}
    \int_{\tau}^{\tau+\delta}Q_{\eps}(s)ds \leq \frac{C(\tau_0)}{|-\ln \delta |}.
\end{equation}
%Here the constants $C(\tau_0)$ (large) and $c(\tau_0)$ (small) are independent of $\eps$.
\end{proposition} 

We stress that the constant $C(\tau_0)$ depends on $\tau_0$ only through the second moment of~$n_{\eps}(\tau_0,v)$.   % $c(\tau_0)$ and 
\\

Recall that $Q_{\eps}(\tau)=1/N_{\eps}(t)$ and $dt=\frac{1}{N_{\eps}(t)}d\tau$. Therefore the right hand side in \eqref{est_finer_Q} corresponds to the time duration in $t$-timescale. This motivates us to return to the $t$-timescale to prove the following counterpart of \eqref{est_finer_Q}

\begin{proposition}\label{prop:t-lower-N}
Assume \eqref{as:ID} and fix a $t_0>0$. For $0<\eps<1$ and a time duration $0<t<1/2$, we have an integrated-in-time lower bound for the firing rate in $t$-timescale in {\eqref{eq:nif-random}},
\begin{equation}\label{est_t-lower-N}
        \int_{t_0}^{t_0+  t}N_{\eps}(s)ds \geq  \exp(-C/ t),
\end{equation} 
where the constant $C>0$ only depends on the second moment of the density $p_{\eps}(t_0,\cdot)$.% at time $t=t_0$.% (t_0,v)$.
\end{proposition}

We first show how to go from Prop.~\ref{prop:t-lower-N} to Prop.~\ref{prop:finer-Q}.

\begin{proof} [Proof of Prop.~\ref{prop:finer-Q}]
Note that as $N_{\eps}\geq 0$, we can extend \eqref{est_t-lower-N}  from time interval $(0,1/2)$ to all $t>0$ as 
\begin{equation}\label{extend-tmp}
        \delta:= \int_{t_0}^{t_0+  t}N_{\eps}(s)ds \geq \begin{cases}
            \exp(-C/ t),\quad 0<t<\frac{1}{2},\\
            \exp(-2C),\quad t\geq 1/2.
        \end{cases} 
\end{equation} The left hand side, denoted as $\delta$, is a time duration in the $\tau$ timescale. If we restrict to the case when $\delta< \exp(-2C)=:c$, then we are in the first regime of \eqref{extend-tmp}, which gives
\begin{equation}
         \exp(-C/ t)\leq \delta.
\end{equation} 
This is equivalent to
\begin{equation}\label{tmp-Q-fine}
        t\leq \frac{C}{|-\ln \delta|}.
\end{equation}
Note that the time duration $t$ corresponds to the integral $\int_{\tau_0}^{\tau_0+\delta}Q_{\eps}(s)ds$ in the new timescale. Therefore, from \eqref{tmp-Q-fine} we deduce \eqref{est_finer_Q} with $0<\delta<c$. {We can extend the result to $0<\delta\leq\frac{1}{2}$ by dividing a larger interval into smaller ones and possibly enlarging $C$ by a factor of $(\frac{1}{2c}+2)$.} Note that the dependence of the constant $C$ on second moment can be transformed to that on time $\tau_0$ thanks to Prop.~\ref{prop:second-moment}. %c,\; 
\end{proof}

\subsection{Proof of Prop. \ref{prop:t-lower-N}}\label{subsec:pf-prop3}

We recall the random discharge system in $t$-timescale~\eqref{eq:nif-random}
\beq \label{eq:peps}
\begin{cases}
     \p_t p_\eps(t,v) + \p_v [(-v  + bN_\eps(t)) p_\eps] - a \p_{vv} p_\eps +  \phi_\eps p_\eps = N_\eps(t)\delta_{V_R} , \quad t \geq 0 , \; v\in \R,
 \\[5pt]
  N_\eps (t) :=\int_\R \phi_\eps(v) p_\eps (t,v) dv 
 \\[5pt]
  p_\eps (t=0,v)= n^0 (v),\quad v\in \R.
\end{cases}
\eeq
Without loss of generality, we set $t_0=0$. 

\begin{proof}[Proof of Prop. \ref{prop:t-lower-N}]
The goal is to make a comparison to a Fokker-Planck equation for which analytical calculations are tractable, namely
\beq 
\begin{cases}
     \p_\tau \tilde{p}_\eps(t,v) + \p_v [(-v  + bN_\eps(t)) \tilde{p}_\eps] - a \p_{vv} \tilde{p}_\eps = 0, \quad t \geq 0 , \; v\in \R,
 \\[5pt]
\tilde{p}_{\eps} (t=0,v)= n^0 (v),\quad v\in \R.
 \end{cases}
\label{FP-OU-recall}
\eeq
To this end, we introduce an auxiliary problem similar to \cite{JGLiuZZ21,JGLiuZZ22} 
 \beq 
\begin{cases}
     \p_t p_\eps^{\text{not}}(t,v) + \p_v [(-v  + bN_\eps(t)) p_\eps^{\text{not}}] - a \p_{vv} p_\eps^{\text{not}} +  \phi_\eps p_\eps^{\text{not}} =0 , \quad t \geq 0 , \; v\in \R,
 \\[5pt]
p_\eps^{\text{not}} (t=0,v)= n^0 (v),\quad v\in \R.
 \end{cases}
\label{IF_not}
\eeq  
Compared with \eqref{eq:peps}, \eqref{IF_not} does not have the reset term and thus $p_{\eps}$ in \eqref{eq:peps} is a supersolution of the linear equation \eqref{IF_not}. Therefore, we have a pointwise comparison%$p_\eps^{\text{not}}$ is a subsolution of \eqref{eq:peps}. Therefore, we have a pointwise comparison
\begin{equation}\label{pointwise-1}
    p_\eps^{\text{not}}(t,v)\leq p_\eps(t,v),\qquad t\geq0,\quad v\in\R,
\end{equation}
 which implies
\begin{equation}\label{bd-N}
    N_\eps^{\text{not}}(t): =\int_{\R} \phi_\eps p_\eps^{\text{not}}(t,v)dv\leq \int_{\R} \phi_\eps p_\eps(t,v)dv= N_\eps(t).
\end{equation} Compared with \eqref{FP-OU-recall}, \eqref{IF_not} has an additional loss term and thus $\tilde{p}_\eps(t,v)$ in \eqref{FP-OU-recall} is also a supersolution of \eqref{IF_not}. Therefore, we have another  pointwise comparison
\begin{equation}\label{pointwise-2}
         p_\eps^{\text{not}}(t,v)\leq \tilde{p}_\eps(t,v) ,\qquad t\geq0,\quad v\in\R
\end{equation}% which

We now estimate $ N_\eps^{\text{not}}(t)$. Integrating Eq. \eqref{IF_not} in $v$, we find that the total mass is decreasing
\begin{equation*}%\label{mass-no}
    \frac{d}{dt}\int_{\R} p_\eps^{\text{not}}(t,v)dv=-\int_{\R} \phi_\eps p_\eps^{\text{not}}(t,v)dv=: -N_\eps^{\text{not}}(t),
\end{equation*} 
and, as the initial mass is $1$, we obtain
\begin{align}%\label{bd-loss}
   \int_0^{t}N_\eps^{\text{not}}(s)ds &=   1-\int_{\R} p_\eps^{\text{not}}(t,v)dv \notag
   \\
   &=\int_{\R} [(\tilde{p}_{\eps}(t,v) - p_\eps^{\text{not}}(t,v)]dv 
   \geq \int_{V_F}^\infty  [\tilde{p}_{\eps}(t,v) - p_\eps^{\text{not}}(t,v)]dv,%= \int_{V_F}^\infty \tilde{p}_{\eps}(t,v) dv- \eps N_\eps^{\text{not}}(t).\label{pf-sec3-tmp}
\end{align} where in the last line we first used that $\tilde{p}_{\eps}(t,\cdot)$ is always of integral one thanks to the mass conservation in \eqref{FP-OU-recall}, and then used the pointwise comparison \eqref{pointwise-2}. Substituting the definition of $N_{\eps}^{\text{not}}$ \eqref{bd-N}, we obtain
\begin{equation}
     \int_0^{t}N_\eps^{\text{not}}(s)ds\geq \int_{V_F}^\infty \tilde{p}_{\eps}(t,v) dv- \eps N_\eps^{\text{not}}(t).\label{pf-sec3-tmp}
\end{equation}

Therefore, we conclude that 
\begin{equation*}
    \frac{d}{dt}\left(e^{t/\eps}\int_0^{t}N_{\eps}^{\text{not}}(s)ds\right)\geq \frac{1}{\eps}e^{t/\eps}\,\int_{V_F}^\infty \tilde{p}_{\eps}(t,v) dv, 
\end{equation*} 
and, recalling \eqref{bd-N}, we finally arrive at
\begin{equation}\label{S2-conclusion}
    \int_0^{t}N_{\eps}(s)ds\geq \int_0^{t}N_{\eps}^{\text{not}}(s)ds \geq \int_0^t\frac{1}{\eps}e^{(s-t)/\eps}\, \int_{V_F}^\infty \tilde{p}_{\eps}(s,v) dvds.
\end{equation}
Combined with Lemma \ref{lem:OUbis} below, we derive that 
\begin{align}
        \int_0^{t}N_{\eps}(s)ds&\geq  \int_{\frac t 2}^t\frac{1}{\eps}e^{(s-t)/\eps}\, e^{-C/s}ds \geq e^{-2C/t} \int_{\frac t 2}^t\frac{1}{\eps}e^{(s-t)/\eps}\, ds \notag
        \\&= e^{-2C/t} (1-e^{-\frac{t}{2\eps}}), \notag
\end{align} 
which concludes the proof of Prop. \ref{prop:t-lower-N}. 
\end{proof}

\begin{lemma} \label{lem:OUbis}
Assuming $n^0 \geq 0$ satisfies $\int n^0=1$ and $\int v^2 n^0(v)dv \leq C_0$, then there is a constant $C$ independent of $N_\eps\geq0$ such that the solution of \eqref{FP-OU-recall} satisfies $\int_{V_F}^\infty \tilde{p}_{\eps}(t_0,v) dv \geq e^{-\frac {C} {t_0}}$ for $0<t_0< \frac 12$.
\end{lemma}

\begin{proof}
    \textit{Step 1. A subsolution to the dual problem.} 
    We consider the following auxillary problem
    \begin{align} \label{OUbis:sol}\begin{cases}
- \partial_t \psi (t,v)  + v \partial_v \psi - a \partial_{vv} \psi = 0, \qquad t\in (0,t_0), \; v \in \R, 
\\
\psi (t_0,v)= \ind{v \geq V_F}.
\end{cases}
\end{align} We claim that $\psi$ is a subsolution of the backward equation of \eqref{FP-OU-recall}, satisfying \begin{align} \label{OUbis:supersol} \begin{cases}
- \partial_t \psi (t,v) - (-v+b N_{\eps}(t)) \partial_v \psi - a \partial_{vv} \psi \leq 0, \qquad t\in (0,t_0), \; v \in \R, 
\\
\psi (t_0,v)= \ind{v \geq V_F}.
\end{cases}
\end{align} Indeed, it follows from that $b N_{\eps}(t) \partial_v \psi \geq 0$, because $w:=\partial_v \psi \geq 0$ since it satisfies
\begin{align*} \begin{cases}
- \partial_t w (t,v)  + v \partial_v w +w  - a \partial_{vv} w = 0, \qquad t\in (0,t_0), \; v \in \R, 
\\
w (t_0,v)=\delta_{v=V_F}\geq 0.
\end{cases}
\end{align*}Then by duality, we have
\begin{align} \label{OUbis:dual}
\int_{V_F}^\infty \tilde{p}_{\eps}(t_0,v) dv = \int_\R \tilde{p}_{\eps} \psi(t_0,v) dv \geq \int_\R n^0(v) \psi(0,v) dv.
\end{align}

\noindent\textit{Step 2. Using the second moment.} It remains to estimate the left hand side of \eqref{OUbis:dual}. We first take $A>0$ large such that $A^2>2C_0$ and $A>|V_F|+2$ (to be used later) and deduce 
\begin{align}\label{est:OUbis1}
\int_{-A}^\infty n^0= 1- \int_{-\infty}^{-A}  n^0 \geq 1- \int_{-\infty}^{-A} \frac{v^2}{A^2} n^0 \geq 1- \frac{C_0}{A^2}= \frac 1 2.
\end{align} Therefore, using that $\psi,\p_v\psi\geq 0$ and \eqref{OUbis:dual} we derive
\begin{align}
    \int_{V_F}^\infty \tilde{p}_{\eps}(t_0,v) dv&\geq \int_{-A}^{+\infty}n^0(v) \psi(0,v) dv\\&\geq  \int_{-A}^{+\infty}n^0(v) \psi(0,-A) dv\geq \frac{1}{2}\psi(0,-A).\label{oubis:tmp}
\end{align}
\noindent\textit{Step 3. Estimate $\psi$. } Indeed, the solution of \eqref{OUbis:sol} can be computed as follows: Define the time variable $s=a(t)$ by 
\[
\frac{da(t)}{dt}=-e^{2(t-t_0)}, \qquad a(t_0)=0, \quad \text{and set} \quad S=a(0)>0,
\] and note that $t_0/e\leq t_0e^{-2t_0} \leq S \leq t_0 $ as $t_0<1/2$. Consider
\[
{\tilde \psi}(a(t),v):= \psi(t, ve^{t_0-t}),
\] which satisfies
\begin{align} \begin{cases}
\partial_s {\tilde \psi} (s,v) - a \partial_{vv} {\tilde \psi} = 0, \qquad s \in (0,S), \; v \in \R, 
\\
{\tilde \psi} (0,v)= \ind{v \geq V_F}.
\end{cases}
\end{align} As a consequence, we have
\begin{equation}
    {\tilde \psi} (S,v)= \frac{C}{\sqrt{S}} \int_{V_F}^\infty e^{-\frac{|v-w|^2}{4aS}}dw,\quad\text{and thus
}\quad \psi(0,v)= \frac{C}{\sqrt{S}} \int_{V_F}^\infty e^{-\frac{|ve^{-t_0}-w|^2}{4aS}}dw.
\end{equation} In particular, recalling \eqref{oubis:tmp} we deduce that
\begin{align*}
    \int_{V_F}^\infty \tilde{p}_{\eps}(t_0,v) dv\geq\frac{1}{2}\psi(0,-A)&=\frac{C}{\sqrt{S}} \int_{V_F}^\infty e^{-\frac{|Ae^{-t_0}+w|^2}{4aS}}dw\\&\geq \frac{C}{\sqrt{S}} \int_{V_F}^{V_F+1} e^{-\frac{|Ae^{-t_0}+w|^2}{4aS}}dw\geq \frac{C}{\sqrt{S}} e^{-\frac{|Ae^{-t_0}+V_F+1|^2}{4aS}}\geq e^{-\frac{C}{t_0}},
\end{align*} where in the last step we used $1/e\leq S/t_0\leq 1$.  Thus the proof of Lemma \ref{lem:OUbis} is complete.

\end{proof}

Thanks to diffusion we intuitively know $\int_{V_F}^\infty \tilde{p}_{\eps}(t_0,v) dv$ shall be positive, and Lemma \ref{lem:OUbis} gives a quantitative bound.
 \begin{remark}
 A probabilistic proof of Lemma \ref{lem:OUbis} is given in Appendix \ref{sec:OU-prob-proof}, using stochastic differential equations (Ornstein–Uhlenbeck process). 
 \end{remark}

\subsection{Remarks on the auxiliary problems}\label{subsec:summary-auxiliary}

     In the proof of Prop.~\ref{prop:t-lower-N}, we introduced $\tilde{p}_{\eps}$ which solves the more standard Fokker-Planck equation \eqref{FP-OU-recall}. Compared to \eqref{FP-OU-recall}, there are two additional effects in the random discharge problem \eqref{eq:peps}: spike (loss) and reset (gain), which make a direct comparison principle between \eqref{eq:peps} and \eqref{FP-OU-recall} unavailable. This motivates us to introduce $p_{\eps}^{\text{not}}$, which solves \eqref{IF_not} and satisfies the comparisons \eqref{pointwise-1} and \eqref{pointwise-2}.

     Indeed, a further decomposition of $p_{\eps}$ holds  as in \cite{JGLiuZZ21,JGLiuZZ22}
\begin{equation}\label{decompose-t}
    p_{\eps}(t,v)=p_{\eps}^{\text{not}}(t,v)+p_{\eps}^{\text{spike}}(t,v).
\end{equation}  
Here $p_{\eps}^{\text{not}}(\tau,v)$ represents the population that has not spiked yet and $p_{\eps}^{\text{spike}}(\tau,v)$ is the part that has spiked at least once. The latter satisfies a system with zero initial data
      \beq 
\begin{cases}
     \p_t p_\eps^{\text{spike}}(t,v) + \p_v [(-v  + bN_\eps(t)) p_\eps^{\text{spike}}] - a \p_{vv} p_\eps^{\text{spike}} +  \phi_\eps p_\eps^{\text{spike}} =N_{\eps}(t)\delta_{V_R} , \quad t \geq 0 , \; v\in \R,
 \\[5pt]
p_\eps^{\text{spike}} (t=0,v)=0,\quad v\in \R.
 \end{cases}
\label{IF_spike}
\eeq 
We have 
\begin{equation}
    N_{\eps}(t)=N_\eps^{\text{not}}(t)+N_\eps^{\text{spike}}(t),\qquad N_\eps^{\text{spike}}(t):=\int_{\R} \phi_\eps p_\eps^{\text{spike}}(t,v)dv,
\end{equation}
\begin{equation}
    \frac{d}{dt}\int_{\R} p_\eps^{\text{spike}}(t,v)dv=N_\eps^{\text{not}}(t)=-\frac{d}{dt}\int_{\R} p_\eps^{\text{not}}(t,v)dv,
\end{equation} and 
\begin{equation}
    \int_{\R} [p_\eps^{\text{not}}(t,v)+ p_\eps^{\text{spike}}(t,v)] dv=1.
\end{equation}
%\end{remark}

In Section~\ref{sec:L2} we will also introduce another auxiliary problem $\bar{p}_{\eps}$, satisfying \eqref{system:redL2}
\begin{equation*}
    \p_t \bar{p}_{\eps} +\p_v  [(-v +b N_{\eps}(t)) \bar{p}_{\eps}] -a\p_{vv} \bar{p}_{\eps} = N_{\eps} (t)\delta_{V_R}(v),\qquad \bar{p}_{\eps}(0,v)=0.
\end{equation*} This system does not have the loss term, compared to \eqref{IF_spike}. Therefore, we have another pointwise comparison~$p_{\eps}^{\text{spike}}\leq\bar{p}_{\eps}$.

For readers' convenience we remark on the later usages of these auxiliary problems. The decomposition \eqref{decompose-t} will be used in Section~\ref{subsec:red}.  Besides, $\bar{p}_{\eps}$ plays an important role in Section~\ref{sec:L2}, where we will also detail the motivations to introduce it.

\section{The bound $n_\eps \in L^\infty_{\tau,\text{loc}} ( L^2_v)$}
\label{sec:L2}

 The controls of $n_{\eps}$ so far only guarantee weak limits in the space of measures, not integrable functions. In particular, the possibility that a Dirac mass is formed at $V_F$ is not ruled out, making it ambiguous to define $\int_{V_F}^{+\infty}ndv$, which is an important quantity in the limit. This motivates us to prove the following $L^2$ estimate.
\begin{proposition}\label{prop:L2tau}
    Assume \eqref{as:ID}. Then, for all $\tau_0>0$, there exists a constant $C(\tau_0)>0$ such that, for all $\varepsilon>0$, {the solution of \eqref{IFnew}} $n_{\eps}$ satisfies
\beq \label{eq:L2-sec4}
 \int_{\R}n_{\eps}^{2}(\tau,v)dv\leq C(\tau_0), \qquad 0\leq \tau\leq \tau_0.
\eeq
\end{proposition}

We emphasize that, despite its seemingly simple statement, Prop.~\ref{prop:L2tau} involves a subtle interplay of various mechanisms. We will first explain these intuitions before presenting the proof.

\subsection{Observations and Intuitions}\label{sec:intuition-L2}

 To explain the intuitions we recall \eqref{IFnew} for convenience
\begin{equation*}
     \p_\tau n_\eps(\tau,v) + \p_v [(-v Q_\eps(\tau) + b) n_\eps] - a Q_\eps(\tau) \p_{vv} n _\eps = \delta_{V_R}(v)-Q_\eps(\tau)  \phi_\eps n_\eps , \quad \tau \geq 0 , \; v\in \R.
\end{equation*}
The main difficulty towards a uniform-in-$\eps$ $L^2$ estimate is the Dirac source term $\delta_{V_R}(v)$ in the right hand side. It physically corresponds to the reset of the voltage after the spike. 

First, we note that the result cannot be directly derived using the parabolic regularization effect from the diffusion term $aQ_{\eps}(\tau)\p_{vv}n_{\eps}$, since the diffusion coefficient $aQ_{\eps}(\tau)$ does not have a uniform-in-$\eps$ positive lower bound. Indeed, $aQ_{\eps}(\tau)$ can degenerate as $\eps\rightarrow0^+$ (c.f. \eqref{lower-bound-Q}), which corresponds to the blow-up of the firing rate in the limit. This motivates us to consider the following toy problem, obtained by setting $Q\equiv0$ in \eqref{eq:Qnif} and neglecting the absorption term 
\begin{equation}\label{toy-1-tau}
    \p_{\tau}m_1+b\p_{v}m_1=\delta_{V_R},\quad \tau\geq0,v\in\R,\qquad\qquad m_1(\tau=0,v)=0.
\end{equation} There is no diffusion in \eqref{toy-1-tau}, but its solution can be computed explicitly as
\begin{equation}
    m_1(\tau,v)=\int_0^{\tau}\delta_{V_R}(v-bs)ds=\frac{1}{b}\mathbb{I}_{V_R<v<V_R+b\tau},\quad \tau>0,v\in\R.
\end{equation} The solution is not only in $L^2_v$ but indeed in $L^{\infty}_v$! Thanks to the transport, Dirac masses starting from different times \textit{disperse} in space (i.e. they don't concentrate), which allows for this $L^{\infty}$ bound.  

The analysis above shows that the toy problem \eqref{toy-1-tau} has a mechanism that regularizes the Dirac mass, in the absence of diffusion. Next, we examine whether such a mechanism persists for the full equation \eqref{IFnew}. To this end, it is convenient to work in $t$ timescale and consider 
\begin{equation}\label{toy-1-t}
\p_{t}q_1+bN(t)\p_{v}q_1=N(t)\delta_{V_R},\quad t\geq0,v\in\R,\qquad\qquad q_1(t=0,v)=0.
\end{equation} By a change of time $d\tau=N(t)dt$ from \eqref{toy-1-tau}, we know the solution of \eqref{toy-1-t} is also bounded in $L^{\infty}$ for $N(t)\in L^1_+$, 

To understand the effects of other terms in the full equation, we extend \eqref{toy-1-t} to the following two toy problems
\begin{equation}\label{toy-2-t}
\p_{t}q_2+\p_{v}((-v+bN(t))q_2)=N(t)\delta_{V_R},\quad t\geq0,v\in\R,\qquad\qquad q_2(t=0,v)=0,
\end{equation} 
\begin{equation}\label{toy-3-t}
\p_{t}q_3+bN(t)\p_{v}q_3=a\p_{vv}q_3+N(t)\delta_{V_R},\quad t\geq0,v\in\R,\qquad\qquad q_3(t=0,v)=0.
\end{equation} 
In \eqref{toy-2-t} we include the relaxation term ``$-v$'' in the drift. Then the drift can degenerate at $v=V_R$ if $bN(t)\equiv V_R$. In this case, the Dirac mass generated at $V_R$ accumulates and $q_2$ can itself become a Dirac mass, given by
\begin{equation}\label{solu-Dirac}
    q_2(t,v)=\int_0^{t}N(s)ds\delta_{V_R}(v)=\frac{tV_R}{b}\delta_{V_R}(v),
\end{equation} 
which of course is not in any $L^p_v$! This singular solution only appears when $V_R>0$ since $bN\geq0$. For the full equation, we may avoid such a singularity using the diffusion term. 

Eq.~\eqref{toy-3-t} is obtained by adding the diffusion term to \eqref{toy-1-t}. Surprisingly, for \eqref{toy-3-t} a $L^{\infty}_v$ estimate can not be expected, in contrast to the pure transport case \eqref{toy-1-tau}-\eqref{toy-1-t}. Consider the following example: for $a=\frac 1 2$, fixed $T>0$, we choose $N(t)= \frac{2}{\sqrt{T-t}}$ and therefore $\int_t^T N = \sqrt{T-t}$, and
\begin{align}\label{solu-td}
q_3(T,V_R)&= \int_0^{T}\frac{C}{\sqrt{T-s}}\exp \left(-\frac{(b\int_{T-s}^{t}N(u)du)^2}{T-s}\right)N(s)ds=\int_0^{T}\frac{C}{T-s}ds =+\infty.
\end{align} 
Diffusion and transport fights against each other. A reason behind is that for a given $x$, the Gaussian density ($\approx\frac{1}{\sqrt{t}}\exp(-\frac{x^2}{t})$) at $x$ with variance $t$  is not monotone in $t$: it first increases and achieves the maximum when $x\approx \sqrt{t}$ and then decreases. Such a blow-up in $L^{\infty}$ reflects that the maximum of different Gaussians starting from different times, can localize at the same spatial point, if the transport by $N$ is chosen properly as in \eqref{solu-td}.

To summarize, $b>0$ is essential for the integrability especially if $N$ is large, as in \eqref{toy-1-tau}-\eqref{toy-1-t}. When $V_R>0$, the transport velocity can degenerate when $bN=V_R$. In that case we may expect the diffusion from $a>0$ to help. Nevertheless, the diffusion and the transport can interact in a subtle way as in the example \eqref{solu-td}. This motivates us to look for a $L^2$ estimate instead of a $L^{\infty}$ one. 

Next, we introduce an auxiliary problem which combines \eqref{toy-2-t} and \eqref{toy-3-t}, a careful study of which will be the core of proof of the $L^2$ estimate.

\begin{remark} Physically the condition $V_R>0$ means that $V_R>V_L+I$, where $V_L$ is the leaky voltage and $I$ is the strength of an external input. We have followed a convention in some mathematical literature to assume $V_L+I=0$ without loss of generality, since otherwise we can  translate in $v$ and use e.g. $V_R-(V_L+I)$ in place of $V_R$.
\end{remark}

\subsection{Reduction to a simpler equation}\label{subsec:red}

To prove the $L^2$ estimate as stated in Prop.~\ref{prop:L2tau}, we first make a reduction to a simpler problem. Consider the following equation in $t$-timescale with zero initial data
\beq \label{system:redL2}\begin{cases}
\p_t \bar{p}_{\eps} +\p_v  [(-v +b N_{\eps}(t)) \bar{p}_{\eps}] -a\p_{vv} \bar{p}_{\eps} = N_{\eps} (t)\delta_{V_R}(v),
\\
\bar{p}_{\eps}(0,v)=0.
\end{cases} \eeq Here $N_{\eps}(t)$ is the firing rate of  the random discharge problem in $t$-timescale with a given initial data, as in \eqref{eq:peps}. We note that  $N_{\eps}(t)$ is a non-negative function which is in $L^1_{loc}$ by \eqref{bd-N-eps}. Compared to~\eqref{eq:peps}, in~\eqref{system:redL2} the loss term is removed and the initial data is set to zero, but all the other terms are kept. It is simpler than \eqref{eq:peps} but more complicated than the toy problems discussed in Section \ref{sec:intuition-L2}.

The following proposition for $\bar{p}_{\eps}$ is the key towards Prop.~\ref{prop:L2tau}. 

\begin{proposition} \label{prop:redL2}
For $\bar{p}_{\eps}$ defined in \eqref{system:redL2}, where $N_{\eps}(t)\geq 0$ is taken from \eqref{eq:peps} with an initial data satisfying \eqref{as:ID}, we have
\begin{equation}\label{estimate-redL2}
\int_\R \bar p_{\eps}(t,v)^2 dv \leq C\left(T\right)\int_0^{T}N_{\eps}(s)ds, \qquad \forall \; t\in [0,T].
\end{equation}
\end{proposition}

Before proving Prop.~\ref{prop:redL2}, we first show how it implies Prop.~\ref{prop:L2tau}. %
\begin{proof}[Proof of Prop. \ref{prop:L2tau}]
    \textit{Step 1. Estimate in $t$.} Consider $p_{\eps}(t,v)$, the solution of the random discharge model in $t$-timescale \eqref{eq:peps}. We shall make use of the decomposition \eqref{decompose-t} in Section \ref{subsec:summary-auxiliary}, $p_{\eps}=p_{\eps}^{\text{not}}+p_{\eps}^{\text{spike}}$, which was originally introduced in \cite{JGLiuZZ21,JGLiuZZ22}. 

    For $p_{\eps}^{\text{not}}$, it satisfies \eqref{IF_not}, the equation without any Dirac sources. Multiplying \eqref{IF_not} by $p_{\eps}^{\text{not}}$ and integrating by parts, it is standard to prove 
    \begin{equation}
       \frac{d}{dt} \int_{\R} \frac{1}{2}(p_{\eps}^{\text{not}})^2dv+ \int_{\R}a|\p_{v}p_{\eps}^{\text{not}}|^2dv+\int_{\R}\phi_{\eps}(p_{\eps}^{\text{not}})^2dv=\int_{\R} \frac{1}{2}(p_{\eps}^{\text{not}})^2dv,
    \end{equation}which implies thanks to the Gronwall lemma
    \begin{equation}
        \int_{\R} (p_{\eps}^{\text{not}})^2(t,v)dv\leq e^t\int_{\R} (n^0)^2(v)dv\leq C(T),\qquad \forall \; t\in [0,T].
    \end{equation} 

    For $p_{\eps}^{\text{spike}}$ which satisfies \eqref{IF_spike}, we make a comparison with $\bar{p}_{\eps}$, the solution of Eq.~\eqref{system:redL2}. Indeed, $p_{\eps}^{\text{spike}}$ is then a subsolution of \eqref{system:redL2} since the additional term $\phi_\eps p_\eps^{\text{spike}}\geq0$, which is in the left hand side of \eqref{IF_spike}. Therefore we have pointwise $0\leq p_{\eps}^{\text{spike}}\leq \bar{p}_{\eps}$, leading to
    \begin{equation}
        \int_{\R} (p_{\eps}^{\text{spike}})^2(t,v)dv\leq \int_{\R} (\bar{p}_{\eps})^2(t,v)dv\leq C\left(T\right)\int_0^{T}N_{\eps}(s)ds,\qquad \forall \; t\in [0,T].
    \end{equation}  In the last step we used Prop.~\ref{prop:redL2}.    

    All together using the decomposition \eqref{decompose-t} we derive a $L^2$ bound in $t$-timescale
\begin{align}
    \int_{\R} (p_{\eps})^2(t,v)dv&\leq 2 \int_{\R} (p_{\eps}^{\text{not}})^2(t,v)dv+2 \int_{\R} (p_{\eps}^{\text{spike}})^2(t,v)dv\\&\leq C\left(T,\int_0^{T}N_{\eps}(s)ds\right),\qquad \forall \; t\in [0,T].
\end{align}
\noindent\textit{Step 2. Return to $\tau$.} Due to \eqref{def:tau} we have $T=\int_0^{\tau_0}Q_{\eps}(\tau)d\tau$ with $\tau_0=\int_0^{T}N_{\eps}(s)ds$. Then in $\tau$ timescale the above estimate becomes, for $0\leq\tau\leq \tau_0$,
\begin{equation}
   \int_{\R} (n_{\eps})^2(\tau,v)dv\leq  C\left(\int_0^{\tau_0}Q_{\eps}(\tau)d\tau,\tau_0\right)\leq C(\tau_0).
\end{equation}In the last inequality we used the $L^1$ bound $\int_0^{\tau_0}Q_{\eps}(\tau)d\tau\leq C(\tau_0)$ proved in Prop.~\ref{prop:second-moment}. This proves Prop.~\ref{prop:L2tau}.
\end{proof}

\begin{remark}
Step 2. of the above proof shows an advantage of working in $\tau$ timescale: the $L^1$ bound on $N(t)$ is automatically guaranteed. Considerations on the time-$t$ duration are detailed in Section~\ref{sec:back-t}. 
\end{remark}

\subsection{Proof of Prop. \ref{prop:redL2}} 

Note that the dependence on $\eps$ in \eqref{system:redL2} is only through $N_{\eps}\geq0$, an external input which determines the solution $\bar{p}_{\eps}$. For the sake of simplifying notations,  we drop the dependence on $\eps$ in $\bar{p}_{\eps}$ and $N_{\eps}$, and seek an estimate of the form \eqref{estimate-redL2} for a given $N(t)\geq0$ in $L^1_{\text{loc}}$ and $\bar{p}$ satisfying 
\begin{equation}\label{redL2-simplified-notation}
    \p_t \bar{p} +\p_v  [(-v +b N(t)) \bar{p}] -a\p_{vv} \bar{p} = N (t)\delta_{V_R}(v),\quad t>0,v\in\R,\qquad
\bar{p}(0,v)=0,\quad v\in\R.
\end{equation}

\begin{proof}[Proof of Prop. \ref{prop:redL2}] We work with the simplified notations as in \eqref{redL2-simplified-notation}.
    
\noindent \textit{Step 1. Duhamel Representation.}
 The solution of \eqref{redL2-simplified-notation} can be represented by the Duhamel formula
\begin{equation}\label{rep-Duhamel}
    \bar{p}(t,v)=\int_0^{t}N(s)G(s,t,v)ds,\quad t\geq0,v\in \R,
\end{equation}
where $G(s,t,v)$ is the Green function starting from $v=V_R$ at time $s$ for the (inhomogeneous) linear part of~\eqref{redL2-simplified-notation}. More precisely, it satisfies
\begin{equation}
        \p_t G +\p_v  [(-v +b N(t)) G] -a\p_{vv}G =0,\quad t>s,v\in\R,\qquad
G(s, s,v)=\delta_{V_R}(v),\quad s\geq0,v\in\R.
\end{equation} The Green function $G$ has a closed form expression as the probability density of the Gaussian random variable
\begin{equation}\label{G-density}
    G(s,t,v)=\frac{1}{\sqrt{2\pi} \sigma(s,t)}\exp\left(-\frac{(v-V(s,t))^2}{2\sigma^2(s,t)}\right),\quad t>s,v\in\R,
\end{equation} where its mean $V(s,t)$ is given by
\begin{equation}\label{formula-mean}
    V(s,t):= e^{s-t}V_R+b\int_s^{t}e^{u-t}N(u)du,
\end{equation} and the variance is given by
\begin{equation}\label{variance-4}
    \sigma^2(s,t):={2a}\int_s^{t}e^{2(u-t)}du=a(1-e^{-2(t-s)}),\qquad \sigma(s,t)>0.
\end{equation}
\noindent  \textit{Step 2. An identity.} Now we compute the $L^2$ norm of $\bar{p}$. Using \eqref{rep-Duhamel} we have
\begin{align}
\bar{p}^2(t,v)&=\int_0^{t}N(s)G(s,t,v)ds\int_0^{t}N(s)G(s,t,v)ds\\&=\int_0^{t}\int_0^{t}N(s_1)N(s_2)G(s_1, t,v)G(s_2, t,v)ds_1ds_2.
\end{align} Hence, using the Fubini theorem, we derive
\begin{align}
\int_{\R}\bar{p}^2(t,v)dv&=\int_{\R}\int_0^{t}\int_0^{t}N(s_1)N(s_2)G(s_1, t,v)G(s_2, t,v)ds_1ds_2 dv\\&=\int_0^{t}\int_0^{t}N(s_1)N(s_2)\left(\int_{\R}G(s_1, t,v)G(s_2, t,v) dv\right)ds_1ds_2.
\end{align} 
Here appears an integral of a pair of Green functions starting from different times, which can be computed explicitly 
\begin{equation}
    \int_{\R}G(s_1, t,v)G(s_2, t,v) dv=\frac{1}{\sqrt{2\pi(\sigma_1^2+\sigma_2^2)}}e^{-\frac{(V_1-V_2)^2}{2(\sigma_1^2+\sigma_2^2)}},
\end{equation} 
where we use the shorthands 
\begin{equation}\label{shorthand}
    V_i:=V(s_i, t),\quad\sigma_i:=\sigma(s_i, t), \qquad i=1,2.
\end{equation} 
Therefore, we arrive at an expression for the $L^2$ norm
\begin{equation}\label{identity-L2}
\int_{\R}\bar{p}^2(t,v)dv=\int_0^{t}\int_0^{t}N(s_1)N(s_2)\frac{1}{\sqrt{2\pi(\sigma_1^2+\sigma_2^2)}}e^{-\frac{(V_1-V_2)^2}{2(\sigma_1^2+\sigma_2^2)}}ds_1ds_2.
\end{equation}

\noindent  \textit{Step 3. Elementary Preparations.} We shall estimate the $L^2$ norm using \eqref{identity-L2}. As preparations, we note the following elementary facts from the definitions of $V_i$ and $\sigma_i$ \eqref{shorthand}, \eqref{formula-mean} and~\eqref{variance-4}. Firstly using $e^{s_1-t}-e^{s_2-t}=-\int_{s_1}^{s_2}e^{u-t}du$ we have
\begin{align}
    V_1-V_2&=b\int_{s_1}^{s_2}e^{u-t}N(u)du+(e^{s_1-t}-e^{s_2-t})V_R\\&= \int_{s_1}^{s_2}e^{u-t}(bN(u)-V_R)du.\label{formula-VV}
\end{align} We also notice that
\begin{equation}\label{formula-sigma}
   2ae^{-t}\bigl(t-s_i\bigr)\leq \sigma_i^2={2a}\int_{s_i}^{t}e^{2(u-t)}du\leq 2a \bigl(t-s_i\bigr),
\end{equation} which implies
\begin{equation}\label{estimate-s1s2}
    2ae^{-t}\bigl(t-s_1+t-s_2\bigr) \leq\sigma_1^2+\sigma_2^2\leq 2a\bigl(t-s_1+t-s_2\bigr).
\end{equation}

We can rewrite \eqref{identity-L2} using the symmetry between $s_1,s_2$, only treating the part $0\leq s_1\leq s_2\leq t$,
\begin{align}\label{symmetric}
\int_{\R}\bar{p}^2(t,v)dv=2\int_0^{t}N(s_1)\left(\int_{s_1}^{t}N(s_2)\frac{1}{\sqrt{2\pi(\sigma_1^2+\sigma_2^2)}}e^{-\frac{(V_1-V_2)^2}{2(\sigma_1^2+\sigma_2^2)}}ds_2\right)ds_1.
\end{align} 
This identity allows us to conclude 
\begin{equation}\label{stronger-bound}
    \int_{\R}\bar{p}^2(t,v)dv \leq C(T)\int_0^{t}N(s_1)ds_1, \qquad \forall t\in[0,T],
\end{equation}
and thus the result \eqref{estimate-redL2}, if we can show the following uniform-in-$s_1$ bound
\begin{equation}\label{claim}
    \int_{s_1}^{t}N(s_2)\frac{1}{\sqrt{2\pi(\sigma_1^2+\sigma_2^2)}}e^{-\frac{(V_1-V_2)^2}{2(\sigma_1^2+\sigma_2^2)}}ds_2\leq C(T), \qquad \forall s_1\in[0,t],\, t\in[0,T],
\end{equation} 
% Indeed, if \eqref{claim} holds, combining it with \eqref{symmetric} we can derive

% in the case $V_R\leq 0$. %Indeed, in this case we obtain a stronger result as $C_0$ does not depend on time.

It remains to prove the claim \eqref{claim}. We first consider the case $V_R\leq 0$ for which a simpler proof can be given.

\noindent  \textit{Step 4. Decomposition: Case $V_R\leq 0$.}   To show \eqref{claim}, now we fix $s_1\in[0,t]$. Note that when $V_R\leq 0$, we have $bN(u)-V_R\geq0$. Therefore by \eqref{formula-VV} the map $s_2\mapsto (V_1-V_2)\geq0$ is non-decreasing for $s_2\in[s_1,t)$. Therefore, since $s_2\mapsto (\sigma^2_1+\sigma_2^2)$ is decreasing \eqref{formula-sigma}, we deduce that the map 
\begin{equation}\label{mono-tmp}
    s_2\mapsto \frac{(V_1-V_2)^2}{2(\sigma_1^2+\sigma_2^2)}\qquad \text{is (continuous and) non-decreasing for $s_2\in[s_1,t)$.}
\end{equation} { Note that the non-decreasing map \eqref{mono-tmp} ranges from its value at $s_2=s_1$, which is zero, to its value at $s_2=t$. The latter is finite for each fixed $s_1<t$, but can go to infinity when $s_1\rightarrow t^-$. As we are looking for a bound that is uniform in $s_1$, for fixed $s_1$  we define the following smaller intervals thanks to \eqref{mono-tmp}}
\begin{equation}\label{def-alphak}
    [\alpha_k,\alpha_{k+1}):=\{s_2\in[s_1,t),\,\, k\leq \frac{(V_1-V_2)^2}{2(\sigma_1^2+\sigma_2^2)}<k+1\},\qquad k\geq 0, k\in\mathbb{Z},
\end{equation} which gives a decomposition of $[s_1,t)$ as
\begin{equation}
    [s_1,t)=\bigcup_{k= 0}^{\infty}[\alpha_k,\alpha_{k+1}).
\end{equation}

Therefore for fixed $s_1$
\begin{align}
    I:= \int_{s_1}^{t}N(s_2)\frac{1}{\sqrt{2\pi(\sigma_1^2+\sigma_2^2)}}e^{-\frac{(V_1-V_2)^2}{2(\sigma_1^2+\sigma_2^2)}}ds_2&=\sum_{k=0}^{\infty}\int_{\alpha_k}^{\alpha_{k+1}} N(s_2)\frac{1}{\sqrt{2\pi(\sigma_1^2+\sigma_2^2)}}e^{-\frac{(V_1-V_2)^2}{2(\sigma_1^2+\sigma_2^2)}}ds_2\\&=: \sum_{k=0}^{\infty} I_k.\label{decompose-Ik}
\end{align} 

\noindent \textit{Step 5. Estimates: Case $V_R\leq 0$.} Now we estimate the integral $I_k$ defined on each sub-interval $[\alpha_k,\alpha_{k+1})$. First, by the lower bound in \eqref{def-alphak} we deduce
\begin{align}\label{Ik-tmp-1}
    I_k\leq \int_{\alpha_k}^{\alpha_{k+1}}\frac{ N(s_2)}{\sqrt{2\pi(\sigma_1^2+\sigma_2^2)}} e^{-k}ds_2=e^{-k}\int_{\alpha_k}^{\alpha_{k+1}}\frac{ N(s_2)}{\sqrt{2\pi(\sigma_1^2+\sigma_2^2)}} ds_2.
\end{align} Also, recalling the formula \eqref{formula-VV}, by the upper bound in \eqref{def-alphak} we derive for any $s_2\in[\alpha_{k},\alpha_{k+1})$,
\begin{equation*}
   \frac{\int_{s_1}^{s_2}e^{u-t}(bN(u)-V_R)du}{\sqrt{2(\sigma_1^2+\sigma_2^2)}} =\frac{V_1-V_2}{\sqrt{2(\sigma_1^2+\sigma_2^2)}}<\sqrt{k+1},
\end{equation*} Therefore, using $V_R\leq 0$ we obtain
\begin{align}
     be^{-t}\frac{\int_{s_1}^{s_2}N(u)du}{\sqrt{2(\sigma_1^2+\sigma_2^2)}}\leq \frac{\int_{s_1}^{s_2}e^{u-t}bN(u)du}{\sqrt{2(\sigma_1^2+\sigma_2^2)}}\leq \sqrt{k+1},
\end{align} which gives
\begin{equation}
    \frac{\int_{s_1}^{s_2}N(u)du}{\sqrt{2\pi(\sigma_1^2+\sigma_2^2)}}\leq\frac{e^{t}\sqrt{k+1}}{b\sqrt{\pi}},\qquad \forall s_2\in[\alpha_k,\alpha_{k+1}).
\end{equation} By continuity this also holds for $s_2=\alpha_{k+1}$. Now, using that  $s_2\mapsto (\sigma^2_1+\sigma_2^2)$ is decreasing \eqref{formula-sigma}, we deduce in \eqref{Ik-tmp-1}
\begin{align}
    I_k\leq e^{-k}\frac{\int_{\alpha_k}^{\alpha_{k+1}}{ N(s_2)} ds_2 }{{\sqrt{2\pi(\sigma_1^2+\sigma_2^2)|_{s_2=\al_{k+1}}}}}\leq e^{-k}\frac{\int_{s_1}^{\alpha_{k+1}}{ N(s_2)} ds_2 }{{\sqrt{2\pi(\sigma_1^2+\sigma_2^2)|_{s_2=\al_{k+1}}}}}\leq  \frac{e^{t}\sqrt{k+1}}{b\sqrt{\pi}} e^{-k},
\end{align} 
which gives the estimate for each term $I_k$. Summing up in $k$ we conclude using \eqref{decompose-Ik}
\begin{equation}\label{final-Vrleq0}
    I\leq \sum_{k=0}^{\infty} \frac{e^{t}}{b\sqrt{\pi}} \sqrt{k+1}e^{-k}=\frac{e^{t}}{b\sqrt{\pi}}\sum_{k=0}^{\infty}  \sqrt{k+1}e^{-k}\leq C(T)<+\infty.
\end{equation} This proves the claim \eqref{claim}, and therefore completes the proof for the case $V_R\leq 0$.

\noindent \textit{Step 6. Case $V_R> 0$.} Now we prove \eqref{claim} for the case $V_R>0$. We still fix $s_1\in[0,t]$. This case is more subtle as we no longer have the monotonicity in \eqref{mono-tmp}. Nevertheless, we still define the (measurable) sets
\begin{equation}
    A_k:=\{s_2\in[s_1,t]:\, k\leq\frac{(V_1-V_2)^2}{2(\sigma_1^2+\sigma_2^2)}\leq k+1\},\quad k\geq 0,k\in\mathbb{Z},
\end{equation} which gives a decomposition of $[s_1,t]$ as $[s_1,t]=\bigcup_{k= 0}^{\infty}A_k$. Therefore, for fixed $s_1$ we obtain
\begin{align}
    I:= \int_{s_1}^{t}N(s_2)\frac{1}{\sqrt{2\pi(\sigma_1^2+\sigma_2^2)}}e^{-\frac{(V_1-V_2)^2}{2(\sigma_1^2+\sigma_2^2)}}ds_2&\leq \sum_{k=0}^{\infty}\int_{A_k} N(s_2)\frac{1}{\sqrt{2\pi(\sigma_1^2+\sigma_2^2)}}e^{-\frac{(V_1-V_2)^2}{2(\sigma_1^2+\sigma_2^2)}}ds_2\\&=: \sum_{k=0}^{\infty} J_k.\label{decompose-Jk}
\end{align} 

To estimate each $J_k$, we first use the lower bound in \eqref{def-alphak} to deduce
\begin{align}\label{Jk-tmp-1}
    J_k\leq \int_{A_k}\frac{ N(s_2)}{\sqrt{2\pi(\sigma_1^2+\sigma_2^2)}} e^{-k}ds_2=e^{-k}\int_{A_k}\frac{ N(s_2)}{\sqrt{2\pi(\sigma_1^2+\sigma_2^2)}} ds_2.
\end{align} Then, we recall \eqref{formula-VV} and use the upper bound in \eqref{def-alphak} to get for any $s_2\in A_k$,
\begin{equation}\label{Jk-tmp2}
   \frac{\int_{s_1}^{s_2}e^{u-t}(bN(u)-V_R)du}{\sqrt{2(\sigma_1^2+\sigma_2^2)}} =\frac{V_1-V_2}{\sqrt{2(\sigma_1^2+\sigma_2^2)}}\leq \sqrt{k+1}.
\end{equation} Now we need to face the difficulties from $V_R>0$. The key observation is that $V_R$ only gives a \textit{bounded} perturbation of the right hand side of \eqref{Jk-tmp2}. More precisely, thanks to \eqref{estimate-s1s2} we derive, recalling $s_1\leq s_2\leq t$,
\begin{align}
         \frac{\int_{s_1}^{s_2}e^{u-t}V_Rdu}{\sqrt{2(\sigma_1^2+\sigma_2^2)}}&\leq \frac{\int_{s_1}^{s_2}V_Rdu}{\sqrt{4ae^{-t}\bigl(t-s_1+t-s_2\bigr)}}= \frac{(s_2-s_1)V_R}{\sqrt{4ae^{-t}\bigl(t-s_1+t-s_2\bigr)}}\\&\leq \frac{(t-s_1)V_R}{\sqrt{4ae^{-t}\bigl(t-s_1+t-s_2\bigr)}} \leq \frac{\sqrt{t-s_1}V_R}{\sqrt{4ae^{-t}}}  \leq \frac{1}{\sqrt{4a}}e^{\frac{1}{2}t}\sqrt{t}V_R.
\end{align} As a consequence, using \eqref{Jk-tmp2} we obtain for $s_2\in A_k$
\begin{align}
     be^{-t}\frac{\int_{s_1}^{s_2}N(u)du}{\sqrt{2(\sigma_1^2+\sigma_2^2)}}&\leq \frac{\int_{s_1}^{s_2}e^{u-t}bN(u)du}{\sqrt{2(\sigma_1^2+\sigma_2^2)}}+\frac{\int_{s_1}^{s_2}e^{u-t}V_Rdu}{\sqrt{2(\sigma_1^2+\sigma_2^2)}}\\&\leq \sqrt{k+1}+\frac{1}{\sqrt{4a}}e^{\frac{1}{2}t}\sqrt{t}V_R,
\end{align} which gives
\begin{equation}
    \frac{\int_{s_1}^{s_2}N(u)du}{\sqrt{2\pi(\sigma_1^2+\sigma_2^2)}}\leq\frac{e^{t}\sqrt{k+1}}{b\sqrt{\pi}}+\frac{1}{\sqrt{4\pi a}}e^{\frac{3}{2}t}\sqrt{t}V_R,\qquad \forall s_2\in A_k.
\end{equation} In particular this holds for $s_2=\alpha_{k+1}^*=:\sup_{s_2\in A_k}s_2$, the rightmost point in $A_k$. Then, we use that  $s_2\mapsto (\sigma^2_1+\sigma_2^2)$ is decreasing in $s_2$ thanks to \eqref{formula-sigma} to deduce in \eqref{Jk-tmp-1}
\begin{align}
    J_k\leq e^{-k}\frac{\int_{A_k}{ N(s_2)} ds_2 }{{\sqrt{2\pi(\sigma_1^2+\sigma_2^2)|_{s_2=\al_{k+1}^*}}}}&\leq e^{-k}\frac{\int_{s_1}^{\alpha_{k+1}^*}{ N(s_2)} ds_2 }{{\sqrt{2\pi(\sigma_1^2+\sigma_2^2)|_{s_2=\al_{k+1}^*}}}}\\&\leq  \frac{e^{t}\sqrt{k+1}}{b\sqrt{\pi}} e^{-k}+\frac{1}{\sqrt{4\pi a}}e^{\frac{3}{2}t}\sqrt{t}V_Re^{-k}.
\end{align} Taking the sum in $k$ we conclude using \eqref{decompose-Jk}
\begin{align}
        I&\leq \sum_{k=0}^{\infty}\left( \frac{e^{t}\sqrt{k+1}}{b\sqrt{\pi}}+\frac{1}{\sqrt{4\pi a}}e^{\frac{3}{2}t}\sqrt{t}V_R \right)e^{-k}\\&=\frac{e^{t}}{b\sqrt{\pi}}\sum_{k=0}^{\infty}  \sqrt{k+1}e^{-k}+\frac{1}{\sqrt{4\pi a}}e^{\frac{3}{2}t}\sqrt{t}V_R\sum_{k=0}^{\infty}  e^{-k}\leq C(T)<+\infty.\label{final-VR>0}
\end{align} This proves the claim \eqref{claim} for the case $V_R>0$, and therefore completes the proof.

\end{proof}

 In the right hand side of \eqref{estimate-redL2},  the term $C(T) \int_0^{T}N(s)ds$ depends linearly on the integral of~$N$. This might not be obvious since $p^2$ depends on $N$ in a nonlinear way (c.f. \eqref{identity-L2}). Moreover, in the proof we achieve more concrete information on the constant $C(T)$ as (c.f. \eqref{final-Vrleq0} and \eqref{final-VR>0})
\begin{equation}
    C(T)=\frac{C_1(T)}{b}+\frac{C_2(T)}{\sqrt{a}}\max(V_R,0),
\end{equation} where $C_1(T)$ and $C_2(T)$ do not depend on $a,b$ or $V_R$. This constant is independent of $a$ if $V_R\leq0$ but will blow up as $a\rightarrow0^+$ if $V_R>0$. And it always blows up if $b\rightarrow0^+$. These are consistent with the intuitive discussions on toy problems in Section \ref{sec:intuition-L2}.

%-------------------------------
\section{Passing to the limit}
\label{sec:limit}
%------------------------------

We are now in a position to pass to the limit and establish Theorem~\ref{thm:limit} and thus also Theorem~\ref{th:main}. For that we first need to complete the proof of Theorem~\ref{thm:uni-in-eps-bounds}.

%----------------------------------------------------------
\subsection{Regularity in time}%and limit equation 
%----------------------------------------------------------

We establish here the strong time continuity statement in Theorem~\ref{thm:uni-in-eps-bounds}, that is

\begin{proposition}\label{prop:equi-continuity}
Assume \eqref{as:ID}. The solution of~\eqref{IFnew} $n_\eps$ satisfies that for any given $\psi(v)$ in $C_b + L^2$, the map 
\begin{equation}\label{map-tau}
         \tau \mapsto \int_{\mathbb{R}}\psi(v)n_{\eps}(\tau,v)dv
\end{equation} 
is equi-continuous with respect to  $\eps$ on every finite interval $[0,\tau_0]$.

\end{proposition}
\begin{proof}
We recall the weak formulation \eqref{weak-test}, for test functions $\psi\in C^2_b(\mathbb{R})$ we have
\begin{equation} \label{weak:forlimit}
    \begin{aligned}
        \frac{d}{d\tau}\int_{\mathbb{R}}\psi(v)n_{\eps}(\tau,v)dv=&\int_{\mathbb{R}}b\p_{v}\psi(v)n_{\eps}(\tau,v)dv +\psi(V_R)\\&+ Q_{\eps}(\tau)\int_{\mathbb{R}}\bigl(a\p_{vv}\psi-v\p_v\psi\bigr)n_{\eps}(\tau,v)dv -\int_{\mathbb{R}} \psi(v) {\cal S}_{\eps}(\tau,v)dv,
    \end{aligned}
\end{equation} where  ${\cal S}_{\eps}(\tau,v)$ is defined in \eqref{def-Seps} with $\int_{\R}{\cal S}_{\eps}(\tau,v)dv=1$. Using the second moment bound \eqref{eq:secondmoment} for $n_\eps$ in Prop.~\ref{prop:second-moment}, we deduce that 
\begin{equation}\label{bd-test}
    \begin{aligned}
    \left| \frac{d}{d\tau}\int_{\mathbb{R}}\psi(v)n_{\eps}(\tau,v)dv\right|\leq &b\|\p_v\psi\|_{\infty}+\|\psi\|_{\infty}\\&+ Q_{\eps}(\tau)\left(a\|\p_{vv}\psi\|_{\infty}+C(\tau)\|\p_{v}\psi\|_{\infty}\right)+\|\psi\|_{\infty}.
\end{aligned}
\end{equation} Thanks to the estimates on $Q_\eps$ in Prop.~\ref{prop:finer-Q}, this proves the (local-in-time) equi-continuity statement for $\psi\in C^2_b$. Using that  $n_\eps (\tau, \cdot) \in \cal P(\R)$, the result extends to $\psi \in C_b$ by a density argument. Similarly using the uniform $L^2$ bound in Prop.~\ref{prop:L2tau}, the result extends to $\psi \in L^2$. Finally, by linearity the result holds for $\psi\in C_b+L^2$.
\end{proof}
%where for the last term we use $Q_{\eps}(\tau)\int_{V_F}^{+\infty}\psi_{\eps}(v)n_{\eps}(\tau,v)dv=1$.
\subsection{Convergent subsequences and the limit equation}\label{subsec:limit}
With usual tools of functional analysis, see~\cite{ReedSimon,Brezis_FA}, and the bounds in Theorem~\ref{thm:uni-in-eps-bounds}, we may extract subsequences (not relabeled) such that for all $\tau_0>0$,  the followings hold.

For $n_{\eps}$, the tightness bound \eqref{eq:secondmoment}, $L^2$ bound \eqref{eq:L2-sec4} and the equal-continuity in Prop.~\ref{prop:equi-continuity} imply that there is a (subsequential) limit $n(\tau)$ satisfying
\begin{equation}
    \label{limit:1} 
    n(\tau)\in\mathcal{P}(\R),\qquad \int_{\R} v^2n(\tau,v)dv+\int_{\R} n^2(\tau,v)dv\leq C(\tau_0),\qquad  0\leq \tau\leq \tau_0.
\end{equation} Here the limit is in the following sense. For any $\psi(v)$ in $C_b + L^2$, as functions in time
\begin{equation}\label{limit:strong-1}
\int_{\mathbb{R}}\psi(v)n_{\eps}(\tau,v)dv \to \int_{\mathbb{R}}\psi(v)n(\tau,v)dv \quad \text{in } C([0,\tau_0]).
\end{equation}

For $Q_{\eps}(\tau)$, by \eqref{eq:est_Qtau} and \eqref{est_finer_Q} we can take a weak limit in measure. More precisely, there exists $Q\in \mathcal{M}_+(0,\tau_0)$ such that for all $\phi(\tau)\in C[0,\tau_0]$
\begin{equation}\label{limit:2-weak}
\int_0^{\tau_0}\phi(\tau)Q_{\eps}(\tau)d\tau\rightarrow \int_0^{\tau_0}\phi(\tau)Q(d\tau).
\end{equation} Moreover, the refined bound \eqref{est_finer_Q} ensures that the limit measure $Q$ does not contain Dirac masses.

Finally the singular term ${\cal S}_{\eps}(\tau,v)$ as defined in \eqref{def-Seps}, which is a probability measure for all $\tau$, satisfies
\begin{equation}\label{limit:3} \begin{cases}
        {\cal S}_{\eps}(\tau,v) \rightharpoonup {\cal S}(\tau,v)\in L^\infty((0,\tau_0);\mathcal{P} (\R)),   \quad \text{weak* in }  \mathcal{M}\big((0,\tau_0)\times\R\big), 
        \\[5pt]
        \int_\R {\cal S}(\tau,dv)=1 \quad \text{for } a.e. \;  \tau >0, \qquad \text{supp } {\cal S}\subseteq [0,\infty)\times[V_F,+\infty),
\end{cases} \end{equation} 
thanks to the tightness implied by \eqref{est_m2rbis}. 
\\

With these convergence results,  we now pass to the limit in the weak formulation~\eqref{weak:forlimit} to obtain that for each test function $\psi(v)\in C^2_b(\mathbb{R})$
\begin{align} \label{WF:limit}
    \frac{d}{d\tau}\int_{\mathbb{R}}\psi(v) n(\tau,v)dv
    =&\int_{\mathbb{R}}b\p_{v}\psi(v)n(\tau,v)dv +\psi(V_R) \notag
    \\ &
    + Q(\tau)\int_{\mathbb{R}}\bigl(a\p_{vv}\psi-v\p_v\psi \bigr)n(\tau,v)dv - \int_{\mathbb{R}} \psi(v) {\cal S}(\tau,v)dv.
\end{align} Here the time derivative $\frac{d}{d\tau}$ is in the weak sense as before. To justify the limit, we only need to take care of the nonlinear terms, which are the products of $Q$ and $n$. The limit of the product can be shown to be the product of limits by a weak-strong convergence argument, using the weak convergence in \eqref{limit:2-weak} and the strong convergence (as functions in time) in \eqref{limit:strong-1}. %By showing \eqref{WF:limit} 

To summarize, we have extracted subsequences from the solution of \eqref{IFnew} $(n_\eps,Q_{\eps},S_{\eps})$ that converge to $(n,Q,S)$ in the weak senses. Moreover, the limit satisfies the weak form of Eq. \eqref{eq:Qnif}.

%----------------------------------------------------------
\subsection{Further properties of $Q(\tau)$ and ${\cal S}(\tau,v)$} %{Further properties of ${\cal S}(\tau,v)$
%----------------------------------------------------------

To further characterize $Q(\tau)$ and ${\cal S}(\tau,v)$, first we aim to derive the complementary relation~\eqref{eq:complementary-relation}
\begin{equation} \label{eq:complementary-relation-repeat}
Q(d\tau)\int_{V_F}^{+\infty}n(\tau,v)dv=0.
\end{equation} Note that $\psi(v)=\mathbb{I}_{v>V_F}$ belongs to  $C_b+L^2$. Therefore the strong limit in time \eqref{limit:strong-1} holds and in the limit we have 
\begin{equation}\label{map-tau-VF}
    \tau \mapsto M(\tau):=\int_{V_F}^{+\infty}n(\tau,v)dv \in C\big([0, \infty)\big).
\end{equation} Thus the right hand side of \eqref{eq:complementary-relation-repeat} is well defined as a measure multiplied by a continuous function. To derive \eqref{eq:complementary-relation-repeat}, we depart from the definition of $Q_{\eps}$ in \eqref{IFnew}
\begin{equation}
    Q_{\eps}(\tau)\int_{V_F}^{+\infty}n_{\eps}(\tau,v)dv=\eps,
\end{equation} and pass to the limit, using the weak convergence of $Q_\eps$ in~\eqref{limit:2-weak} and the strong convergence in \eqref{limit:strong-1} with $\psi(v)=\mathbb{I}_{v>V_F}$. Thus \eqref{eq:complementary-relation-repeat} is established.

The complementary relation \eqref {eq:complementary-relation-repeat} gives information on $Q$ in an implicit way. To state more precise characterizations, since $M(\tau)$ is continuous  we define the sets \begin{align}\label{def-Ibl}
    I_{bl}&:=\{\tau\geq0:M(\tau)>0\}, \qquad \text{an open set},
    \\
    I_{cl}&:=\{\tau\geq0:M(\tau)=0\},\qquad \text{a closed set}.\label{def-Icl}
\end{align} 
Here $I_{bl}$ is for ``blow-up'' and $I_{cl}$ is for ``classical solution''. 

 \begin{remark}
     The set  $I_{cl}$ may contain blow-up points (i.e. $Q(\tau)=0$) therefore it can be strictly larger than the set of times when we can indeed recover classical solutions. This is partially because the complementary relation \eqref{eq:complementary-relation-repeat} does not exclude times when \textit{both} $Q(\tau)$ and $M(\tau)$ are zero. We give such an example in Section \ref{sec:back-t}, see \eqref{steady-state} with $b=V_F-V_R$.
 \end{remark}

%---------------------------------------------------
\subsubsection{The interior of the classical points}\label{subsec:inter-cl}

In the \textit{interior} of the ``classical points'' $\mathring{I}_{cl}\subseteq I_{cl}$ we can go much further. We characterize ${\cal S}$ and recover the standard boundary conditions when $Q>0$.

\begin{proposition}\label{Prop:S-Dirac}
    For $\tau\in \mathring{I}_{cl}$, we have
    \begin{equation}\label{eq:limit-S-cl}
        {\cal S}(\tau,v)=\delta_{V_F}(v),
    \end{equation}
    \begin{equation}\label{eq-boundary-term}
\lim_{\eps\rightarrow0^+}         Q_{\eps}(\tau)n_{\eps}(\tau,V_F)=0,\qquad \lim_{\eps\rightarrow0^+}         \bigl(aQ_{\eps}(\tau)\p_vn_{\eps}(\tau,V_F)+bn_{\eps}(\tau,V_F)\bigr)=1,
    \end{equation} where the convergence holds in the weak sense of measures.
\end{proposition}
In particular, when $Q>0$ in the limit, \eqref{eq-boundary-term} gives the Dirichlet boundary condition for $n$ and the classical definition of the firing rate: 
\begin{equation}
    n(\tau,V_F)=0,\qquad aQ(\tau)\p_vn(\tau,V_F)=1.
\end{equation} 
\begin{proof}We first prove \eqref{eq:limit-S-cl}.
    The definition of $I_{cl}$ implies that $n(\tau,v)\equiv 0$ for $\tau\in I_{cl}$ and $v>V_F$. Therefore, for $\tau$ in the interior $\mathring{I}_{cl}$, we consider the weak formulation \eqref{WF:limit} with $\psi\in C_c^{\infty}(V_F,+\infty)$ such that $\psi\geq 0$ to obtain
    \begin{align}
            0=\frac{d}{d\tau}\int_{\mathbb{R}}\psi(v) n(\tau,v)dv
    = - \int_{\mathbb{R}} \psi(v) {\cal S}(\tau,v)dv\leq 0.
    \end{align}
Therefore the equality holds, which implies
\begin{equation}
    \int_{\mathbb{R}} \psi(v) {\cal S}(\tau,v)dv=0,
\end{equation} 
for all non-negative $\psi\in C_c^{\infty}((V_F,+\infty))$. Therefore we deduce that the support of ${\cal S}$ does not contain any $v>V_F$ for $\tau\in\mathring{I}_{cl}$. Combining this with \eqref{limit:3}, we deduce that
\begin{equation}
    \text{supp } {\cal S}(\tau,v) \cap (\mathring{I}_{cl}\times\R)\subseteq  \mathring{I}_{cl}\times \{V_F\}.
\end{equation} In other words, for $\tau\in\mathring{I}_{cl}$ the measure $ {\cal S}(\tau,v)$ is localized at $V_F$. Therefore, we recall from \eqref{limit:3} that for (almost) every $\tau$ $S(\tau,\cdot)$ is of mass one, to conclude
\begin{equation*}
    S(\tau,v)=\delta_{V_F},\qquad \tau\in\mathring{I}_{cl}.
\end{equation*} 

To prove \eqref{eq-boundary-term}, we first work with the random discharge problem \eqref{IFnew} with $\eps>0$. Multiply \eqref{IFnew} with a test function $\psi\in C^{\infty}_c(\R)$ and integrate on $[V_F,+\infty)$, and we obtain
\begin{align}
\frac{d}{d\tau}\int_{V_F}^{+\infty}n_{\eps}\psi dv+&\int_{V_F}^{+\infty}((-vQ_{\eps}+b)\p_v{\psi}+a\p_{vv}\psi)n_{\eps} dv=\\&a\p_v\psi(V_F)Q_{\eps}n_{\eps}(V_F)+\psi(V_F)(aQ_{\eps}\p_vn_{\eps}(V_F)+bn_{\eps}(V_F))-\int_{V_F}^{+\infty}\psi {\cal S}_{\eps}dv.
\end{align} For $\tau\in\mathring{I}_{cl}$, following the same procedure as in Section \ref{subsec:limit}, we see that the left hand side vanishes as $\eps\rightarrow0^+$. Choosing $\psi$ such that $\psi(V_F)=1$ and $\p_v\psi(V_F)=0$, we obtain using \eqref{eq:limit-S-cl}
\begin{equation}
   \lim_{\eps\rightarrow0^+}         \bigl(aQ_{\eps}(\tau)\p_vn_{\eps}(\tau,V_F)+bn_{\eps}(\tau,V_F)\bigr)=1.
\end{equation} Next, we choose $\psi$ such that $\psi(V_F)=0$ and $\p_v\psi(V_F)=1$ to get
\begin{equation}
    \lim_{\eps\rightarrow0^+}         Q_{\eps}(\tau)n_{\eps}(\tau,V_F)=0.
\end{equation}

% \lim_{\eps\rightarrow0^+}         Q_{\eps}(\tau)n_{\eps}(\tau,V_F)=0,\qquad 

    %n_{\eps}(V_F)[b\psi(V_F)+Q_{\eps}V_F\psi(V_F)-aQ_{\eps}\p_v\psi(V_F)]+\psi(V_F)a\p_vn_{\eps}(V_F)Q_{\eps}
\end{proof}

\subsubsection{Blow-up points}

For $\tau\in I_{bl}$ as defined in \eqref{def-Ibl}, we can directly compute the limits of $Q$ and ${\cal S}$ to find 
\begin{proposition}\label{prop:limit-Ibl}
    For $\tau\in I_{bl}$, we have
\begin{equation}
       Q(\tau)= \lim_{\eps\rightarrow0^+}Q_{\eps}(\tau)=0, \qquad
       S_{\eps}(\tau,\cdot) \rightharpoonup   {\cal S}(\tau,\cdot)=\frac{n(\tau,\cdot)\mathbb{I}_{ \, \cdot \, >V_F}}{\int_{V_F}^{+\infty}n(\tau,v)dv},\quad \text{weakly in } {\cal P}(\R).
\end{equation}
\end{proposition} Indeed, this follows directly from passing the limits in $Q_{\eps}(\tau)=\frac{\eps}{\int_{V_F}^{+\infty}n_{\eps}(\tau,v)dv}$ and ${\cal S}_{\eps}(\tau,\cdot)= \frac{n_{\eps}(\tau,\cdot)\mathbb{I}_{ \, \cdot \, >V_F}}{\int_{V_F}^{+\infty}n_{\eps}(\tau,v)dv}$, due to the strong limit \eqref{limit:strong-1} with $\psi(v)=\mathbb{I}_{v>V_F}$. The fact that $Q$ vanishes on $I_{bl}$ can also be derived directly from the complementary relation \eqref{eq:complementary-relation-repeat}.

\begin{remark}
In the limit $n(\tau,\cdot)$ is defined for every time and is continuous in time as in \eqref{limit:1}-\eqref{limit:strong-1}. For $Q$ and ${\cal S}$, by Prop.\ref{prop:limit-Ibl} they are continuous in time on $I_{bl}$. In general, we do not expect $Q$ and $\cal S$ to be continuous in time on $[0,+\infty)$. 
\end{remark}

Prop.~\ref{prop:limit-Ibl} provides a starting point to further investigate blow-ups next.

%----------------------------------------------------
\section{More properties of the solution}\label{sec:6}

\subsection{More on the blow-up intervals}

For the blow-up open set $I_{bl}$ as defined in \eqref{def-Ibl}, we first recall from Prop.~\ref{prop:limit-Ibl} that 
\begin{equation}\label{QS-blowup}
            Q(\tau)=0,\quad   \text{and}\quad {\cal S}(\tau,\cdot)=\frac{n(\tau,\cdot)\mathbb{I}_{\cdot>V_F}}{\int_{V_F}^{+\infty}n(\tau,v)dv},\qquad \forall \tau\in I_{bl},
\end{equation} 
Using that $Q\equiv0$, we simplify Eq.~\eqref{eq:Qnif} into \eqref{eq:Qnif-bl}
\begin{equation}\label{eq:Qnif-bl-Ibl}
         \frac{\partial n}{\partial \tau} +b\frac{\partial }{ \partial v} n = \delta_{V_R}(v)- {\cal S}(\tau,v) , \qquad \tau \in I_{bl}, \, v \in \R.
\end{equation}

 Consider now a maximal interval $(\tau_1,\tau_2)\subseteq I_{bl}$ with $\tau_2>\tau_1>0$. Recall the definition of $M(\tau)$ in~\eqref{map-tau-VF}. We have, by the continuity of $M(\cdot)$,
 \begin{equation}\label{M-value}
        M(\tau_1)=0,\qquad  M(\tau)>0,\quad  \tau\in(\tau_1,\tau_2),\qquad M(\tau_2)=0\quad \text{if $\tau_2<\infty$}.
\end{equation} 
Denote by $n_{pre}(v):=n(\tau_1,v)$ the pre-blow-up profile, which is supported in $(-\infty, V_F]$. Since the characteristics is moving rightwards ($b>0$),  solving \eqref{eq:Qnif-bl-Ibl} when $\tau\in (\tau_1,\tau_2)$, as a transport equation with sources, gives
\begin{equation} \label{n-blowup-full}
        n(\tau,v)=n_{pre}(v-b(\tau-\tau_1))+\frac{1}{b}\mathbb{I}_{V_R\leq v< V_R+b(\tau-\tau_1)}- \int_{\tau_1}^{\tau}{\cal S}(s,v-b(\tau-s))ds.
\end{equation}
Note that the last integral is well-defined as a $L^2$ function in $v$ thanks to the regularity of ${\cal S}$ in \eqref{QS-blowup}. Moreover, it vanishes for $v\leq V_F$ since ${\cal S}$ is supported on $[V_F,+\infty)$  and $b>0$. 

When $\tau_2<\infty$, $n(\tau_2,v)$ is also supported in $(-\infty, V_F]$, since  $M(\tau_2)=0$, and we infer the post-blow-up profile at $\tau=\tau_2$
\begin{equation}\label{post-profile}
        n_{post}(v):=n(\tau_2,v)=\left(n_{pre}(v-b(\tau_2-\tau_1))+\frac{1}{b}\mathbb{I}_{V_R\leq v< V_R+b(\tau_2-\tau_1)}\right)\mathbb{I}_{v\leq V_F},\qquad  v\in \R.
\end{equation}

%-------------------------------------
\begin{proposition}\label{prop:char-bl}
We have for $M(\tau)$ defined in \eqref{map-tau-VF}
\begin{equation}\label{dM-dtau}
        \frac{d}{d\tau}M(\tau)=bn(\tau,V_F)-1,\quad \tau\in I_{bl}.
\end{equation} 
Notice that $n(\tau,v)$ is not defined pointwise in $v$, but $n(\tau,V_F)$ is  defined a.e. in $\tau$ thanks to~\eqref{n-blowup-full}.

Furthermore, for $\tau\in(\tau_1,\tau_2)$, we can write $M(\tau)$ in terms of $n_{pre}$ via
\begin{equation}\label{expression-M}
        M(\tau) =\begin{dcases}
             \int_{V_F-b(\tau-\tau_1)}^{V_F}n_{pre}(v)dv-(\tau-\tau_1),\qquad 0\leq \tau-\tau_1\leq\frac{V_F-V_R}{b},\\
            \int_{V_F-b(\tau-\tau_1)}^{V_F}n_{pre}(v)dv-\frac{V_F-V_R}{b},\qquad \frac{V_F-V_R}{b}\leq \tau-\tau_1.\\
         \end{dcases}
\end{equation}
\end{proposition}

\begin{remark}\label{rmk:tmp-5.7}

     Effectively the dynamics in $I_{bl}$ is characterized by the pair $(n(\tau,v)\mathbb{I}_{v\leq V_F},M(\tau))$. The details of how the mass is distributed in $v>V_F$, which depend on the specific form of ${\cal S}$, do not play a role in the effective dynamics. 

    Indeed, the proof of Proposition \ref{prop:char-bl} does not use the specific form of ${\cal S}$ given in \eqref{QS-blowup}, only that it is of mass $1$. In particular, if we use another absorption function $\phi_{\eps}(v)=\frac{1}{\eps}(v-V_F)_+$ instead of \eqref{eq:phi}, then    the proposition still holds with ${\cal S}$ given by
    \begin{equation}
        {\cal S}(\tau,\cdot)=\frac{(v-V_F)_+n(\tau,\cdot)}{\int_{V_F}^{+\infty}{(v-V_F)_+n(\tau,v)dv}}.
    \end{equation}
     In other words, in the regularized problem, the absorption profile of mass for  $v>V_F$ does not matter except : i) it has a constant rate $1$; ii) it keeps the non-negativity of $n$.
\end{remark}

\begin{proof} 
    Identity \eqref{dM-dtau} intuitively follows from integrating Eq.~\eqref{eq:Qnif-bl-Ibl} on $[V_F,+\infty)$, yet we do not know if $n(\tau,V_F)$ can be defined pointwise. It can be justified as $n(\tau,V_F)$ can be viewed as an~$L^1$ function in time thanks to \eqref{n-blowup-full}. More precisely, we integrate on $[V_F,+\infty)$ in \eqref{n-blowup-full} to obtain
\begin{align}
        M(\tau)&=\int_{V_F}^{+\infty}\left(n_{pre}(v-b(\tau-\tau_1))+\frac{1}{b}\mathbb{I}_{V_R\leq v< V_R+b(\tau-\tau_1)}- \int_{\tau_1}^{\tau}{\cal S}(s,v-b(\tau-s))ds\right)dv\\&=\int_{V_F}^{+\infty}\left(n_{pre}(v-b(\tau-\tau_1))+\frac{1}{b}\mathbb{I}_{V_R\leq v< V_R+b(\tau-\tau_1)}\right)dv- (\tau-\tau_1),\label{tmp-blowup}
\end{align} 
where we use, since $b>0$ and ${\cal S}(s,\cdot)$ is supported on $[V_F,+\infty)$,
\begin{equation}
        \int_{V_F}^{+\infty}\int_{\tau_1}^{\tau}{\cal S}(s,v-b(\tau-s))dsdv= \int_{\tau_1}^{\tau}\int_{V_F}^{+\infty}{\cal S}(s,v-b(\tau-s))dvds=  \int_{\tau_1}^{\tau}1ds=\tau-\tau_1 .
\end{equation} 
 As $M(\tau_1)=0$ we have $n_{pre}(v)=0$ for $v>V_F$, we may rewrite \eqref{tmp-blowup} as
\begin{equation}
        M(\tau)= \int_{V_F}^{V_F+b(\tau-\tau_1)}\left(n_{pre}(v-b(\tau-\tau_1))+\frac{1}{b}\mathbb{I}_{V_R\leq v< V_R+b(\tau-\tau_1)}\right)dv - (\tau-\tau_1).
\end{equation} 
Writing $v-b(s-\tau_1)=V_F-b(s-\frac{V_F-v}{b}-\tau_1)$ and using a change of variable, we arrive at
    \begin{equation}
        M(\tau)=\int_{\tau_1}^{\tau}\big(bn(s,V_F)-1\big)ds.
    \end{equation}

Finally, \eqref{expression-M} follows from rewriting the first integral in \eqref{tmp-blowup}, using
\begin{align*}
    \int_{V_F}^{+\infty}n_{pre}(v-b(\tau-\tau_1))dv&= \int_{V_F-b(\tau-\tau_1)}^{V_F}n_{pre}(v)dv,\\
    \int_{V_F}^{+\infty}\frac{1}{b}\mathbb{I}_{V_R\leq v< V_R+b(\tau-\tau_1)}dv&=\frac{1}{b}\int_{V_F}^{\max(V_F,V_R+b(\tau-\tau_1))}dv=\max(0,\tau-\tau_1-\frac{V_F-V_R}{b}).
\end{align*} 
Note that the term involving $\mathbb{I}_{V_R\leq v< V_R+b(\tau-\tau_1)}$, which comes from the reset $\delta_{V_R}(v)$, only takes effect when $\tau-\tau_1\geq \frac{V_F-V_R}{b}$, since it takes time to transport from $V_R$ to $V_F$ with velocity $b>0$.

\end{proof}

 Prop.~\ref{prop:char-bl} has some consequences on the blow-up interval $\tau_2-\tau_1$.

\begin{corollary}[Blow-up interval]\label{cor:tau2-tau1}
Let $(\tau_1,\tau_2)\subseteq I_{bl}$ be a connected component of $I_{bl}$ with $\tau_1>0$.
\\
(i) When $0<b<V_F-V_R$, then $\tau_2-\tau_1\leq 1<\frac{V_F-V_R}{b}$.\\ (ii) If $\tau_2<\infty$, then
\begin{equation}\label{eq:tau2-tau1}
        \tau_2-\tau_1=\int_{V_F-b(\tau_2-\tau_1)}^{V_F}n_{pre}(v)dv,
\end{equation} and $\tau_2$ is the infimum value of $\tau>\tau_1$ such that \eqref{eq:tau2-tau1} holds.
\\
(iii) 
    If $\tau_2-\tau_1> \frac{V_F-V_R}{b}$ (and thus $b\geq V_F-V_R$), then $\tau_2=+\infty$.
\\
(iv) Denoting by  $n_{pre}(v):=n(\tau_1,v)$ the pre-blow-up profile, then
\begin{equation}\label{eq:nprevvf}
   \liminf_{\delta\rightarrow0^+}\frac{b\int_{V_F-\delta}^{V_F}n_{pre}(v)dv}{\delta}\geq 1.
\end{equation} 
\end{corollary} 

 As the absorption rate by ${\cal S}$ is one, the duration $\tau_2-\tau_1$ physically quantifies the number of neurons which spike during the blow-up interval. They spike simultaneously at a single blow-up time in the original timescale $t$. Eq.~\eqref{eq:tau2-tau1} gives a characterization of the number in terms of $n_{pre}$, see also similar formulas in \cite{zhang2014distribution,delarue2015particle,dou2022dilating,taillefumier2022characterization}. Eq.~\eqref{eq:nprevvf} implies that the Dirichlet boundary condition at $V_F$ is lost at the blow-up time, which is also known in literature e.g. \cite{hambly2019mckean,taillefumier2022characterization}.

\begin{proof} We begin with (i). When $0<b<V_F-V_R$, if $1<\tau_2-\tau_1$, then we can choose $\tau\in(\tau_1,\tau_2)$ such that $1<\tau-\tau_1<\frac{V_F-V_R}{b}$. By \eqref{expression-M} we have
\begin{align}
        0<M(\tau)=\int_{V_F-b(\tau-\tau_1)}^{V_F}n_{pre}(v)dv-(\tau-\tau_1)\leq 1-(\tau-\tau_1)<0,
\end{align} which is a contradiction.

To prove (iii), when $\tau_2-\tau_1> \frac{V_F-V_R}{b}$ (by the first statement,  this happens only when $b\geq V_F-V_R$), if $\tau_2<\infty$, then we note that both $\tau_2$ and $\tau_1+\frac{V_F-V_R}{b}$ are in the second regime in  \eqref{expression-M}. However, we note that $M(\tau)$ is non-decreasing in $\tau$ in that regime, which implies 
\begin{equation}
        0=M(\tau_2)\geq M\left(\tau_1+\frac{V_F-V_R}{b}\right)>0,
\end{equation} which is a contradiction. For the last inequality, we use \eqref{M-value} and $\tau_1+\frac{V_F-V_R}{b}\in(\tau_1,\tau_2)$.

Next, if $\tau_2<\infty$ then by (iii) we have $\tau_2-\tau_1\leq\frac{V_F-V_R}{b}$. In view of \eqref{expression-M}, (ii) s because $M(\tau_2)=0$ and $M(\tau)>0$ for $\tau\in(\tau_1,\tau_2)$.

Finally, the formula \eqref{eq:nprevvf} is a direct consequence of $M(\tau_1+\delta)>0$ for small $\delta>0$, and the first expression in \eqref{expression-M}.

\end{proof} 

 The critical value $b=V_F-V_R$ is related to the lifespan in $t$ of the solution, which we discuss more below.

%----------------------------------
\subsection{Lifespan in $t$-timescale}
\label{sec:back-t}
%----------------------------------

Theorem~\ref{th:main} 
 provides a global solution in $\tau$ timescale, but it may not be global in $t$ timescale. Using the change of time \eqref{def:tau} the lifespan in $t$ is obtained as 
\begin{equation}\label{def-lifespan}
T^*=\int_0^{+\infty}Q(\tau)d\tau.
\end{equation} Then the solution in $t$ is global if and only if $T^*=+\infty$.

%Similar to the main theorem in

\begin{theorem}[Lifespan in $t$ timescale]\label{th:lifespan} With $T^*$ defined in \eqref{def-lifespan},
\\ (i) when $0<b<V_F-V_R$, we have $T^*=+\infty$,
\\
(ii) when $b\geq V_F-V_R$, there exist examples with finite $T^*$ (we may even have  $T^*=0$).
\end{theorem}

In other words, for $0<b<V_F-V_R$ the solution is always global in $t$, but for $b\geq V_F-V_R$ we have counter-examples. {We note that the definition of lifespan \eqref{def-lifespan} and this dichotomy between values of $b$ resemble those in \cite{dou2022dilating}.}
\begin{proof}
    To prove (i) we use the weak formulation \eqref{WF:limit}. Take a test function $\psi\in C^2(\R)$ such that $0\leq \p_v\psi\leq 1$ 
 for all $v\in \R$ satisfying
    \begin{equation}
\psi\equiv0,\quad v\leq V_R-1,\qquad\p_v\psi \equiv1,\quad v\in[V_R,V_F],\qquad\p_v\psi\equiv0,\quad v\geq V_F+1.\end{equation} By construction $\psi$ is bounded and non-decreasing. The upper bound on $\p_v\psi$ implies
    \begin{equation}
\int_{\mathbb{R}}b\p_{v}\psi(v)n(\tau,v)dv \leq  \int_{\mathbb{R}}bn(\tau,v)dv =b.
\end{equation} Since ${\cal S}(\tau,\cdot)$ is supported on $[V_F,+\infty)$ we compute
\begin{align}
       \psi(V_R)-\int_{\mathbb{R}} \psi(v) {\cal S}(\tau,v)dv\leq \psi(V_R)-\psi(V_F)=-(V_F-V_R). 
\end{align} Note also $\p_v \psi$ is a compactly supported $C^1$ function, which implies $a\p_{vv}\psi-v\p_v\psi\leq C$. Thus we have
\begin{equation}
Q(\tau)\int_{\mathbb{R}}\bigl(a\p_{vv}\psi-v\p_v\psi \bigr)n(\tau,v)dv \leq  CQ(\tau)\int_{\mathbb{R}}n(\tau,v)dv=CQ(\tau).
\end{equation} All together we derive in \eqref{WF:limit}
\begin{equation}
    \frac{d}{d\tau}\int_{\mathbb{R}}\psi(v) n(\tau,v)dv\leq  b-(V_F-V_R)+ CQ(\tau).
\end{equation} Integrating in time, we obtain
\begin{align}
C\int_0^{\tau}Q(s)ds&\geq (V_F-V_R-b)\tau+\int_{\mathbb{R}}\psi(v) n(\tau,v)dv-\int_{\mathbb{R}}\psi(v) n(0,v)dv\\ &\geq (V_F-V_R-b)\tau-C,
\end{align} which goes to infinity as $\tau$ goes to infinity, provided that $b<V_F-V_R$. This proves $T^*=+\infty$.

For (ii) when $b\geq V_F-V_R$, one can immediately  check that the following is a steady state of Eq.~\eqref{eq:Qnif-bl-Ibl} 
\begin{equation}\label{steady-state}
n(v)=\begin{dcases}
      \frac{1}{b}\mathbb{I}_{V_R\leq v\leq V_F},\quad v\leq V_F,\\
      \frac{1}{b}e^{-\frac{1}{bM}(v-V_F)},\quad v>V_F,
    \end{dcases}
\end{equation} with $M=1-\frac{V_F-V_R}{b}>0$ when $b>V_F-V_R$. When $b=V_F-V_R$ we have $M=0$ and the above is understood as $n(v)\equiv0$ for $v>V_F$. 

A more general class of solutions of Eq.~\eqref{eq:Qnif} with $Q\equiv0$ is to take \begin{equation}
    \begin{dcases}
    n(\tau,v)=\frac{1}{b}\mathbb{I}_{V_R\leq v\leq V_F},\qquad \tau\geq0,\, v\leq V_F,\\ M(\tau)=1-\frac{V_F-V_R}{b}\geq0,\qquad \tau\geq0.
        \end{dcases}
\end{equation} Here the profile of $n(\tau,v)$ for $v>V_F$ can be time-dependent and in Eq.~\eqref{eq:Qnif-bl},  ${\cal S}(\tau,v)$ adapts to satisfy the equality for $M(\tau)$, see also Remark.~\ref{rmk:tmp-5.7}. For those examples $T^*=0$.

%    \frac{d}{d\tau}\int_{\mathbb{R}}\psi(v) n(\tau,v)dv=&\int_{\mathbb{R}}b\p_{v}\psi(v)n(\tau,v)dv +\psi(V_R) \notag\\ &+ Q(\tau)\int_{\mathbb{R}}\bigl(a\p_{vv}\psi-v\p_v\psi \bigr)n(\tau,v)dv - \int_{\mathbb{R}} \psi(v) {\cal S}(\tau,v)dv.

\end{proof}

The examples constructed for Theorem~\ref{th:lifespan}-(ii) are related to the ``plateau solutions'' studied in \cite{CiCP-30-820}, and are called ``eternal blow-up'' in \cite{dou2022dilating}. It means every neuron spikes for infinite times at a single time in $t$ timescale, which is unrealistic as it is incompatible with the refractory state.

\section{Conclusions and discussion}\label{sec:7}

To understand the dynamics after blow-up in the integrate-and-fire models for neural assemblies, we study the random discharge model as a regularized problem. Using the dilated timescale $\tau$, we are able to obtain new estimates. Those are fundamental to pass to the strong absorption limit $\eps\rightarrow0^+$ in the nonlinear terms. A global limit equation is derived where the Dirichlet boundary condition is relaxed by a measure ${\cal S}$ which is a Lagrange multiplier to keep $n$ as a probability distribution. As consequences, we obtain different information on the blow-up depending on the critical parameter~$\frac{V_F-V_R}{b}$.

Several questions remain open. Mathematically, we know little about the global behavior of the blow-up and classical sets $I_{bl}$ and $I_{cl}$, though we characterize the dynamics locally. For instance, we don't know if $I_{bl}$ can have infinite number of connected components on a finite interval. These could be potential difficulties towards the uniqueness of the limit solution. See also \cite{delarue2022global,LS_2020,sadun2022global} for results about related models.

From a physical point of view, our model is highly simplified. For example, among all missing physical mechanisms, the refractory period might account for the pathological case when $T^*=0$ in Theorem \ref{th:lifespan}. 

Also, it will be interesting to look into the connections between various continuations after blow-up, for instance those obtained from the finite network of neurons \cite{delarue2015global} and for the case with activity-dependent noise \cite{dou2022dilating,sadun2022global,taillefumier2022characterization}.

%--------------------------
\appendix

\section{A Probabilistic Proof of Lemma \ref{lem:OUbis}}\label{sec:OU-prob-proof}

Here we give a probabilistic proof of Lemma \ref{lem:OUbis}, using the following inhomogenous OU process
\begin{equation}\label{SDE-tildeX}
    \begin{cases}
        d\tilde{X}_{\eps,t}=(-X_t+bN_{\eps}(t))dt+\sqrt{2a}dB_t\quad t>0,\\
        \tilde{X}_{\eps,T=0}=X_0,\\
    \end{cases}
\end{equation} whose Fokker-Planck equation is \eqref{FP-OU-recall}. That is, if initially the distribution of $X_0$ is given by $n^0(v)$, then the evolution of the probability density of $\tilde{X}_{\eps,t}$ is governed by \eqref{FP-OU-recall}. In this way Lemma \ref{lem:OUbis} is reformulated as

\begin{lemma}\label{lem:OU}
    Let $\tilde{X}_{t,\eps}$ as defined in \eqref{SDE-tildeX} with $\mathbb{E}X_0^2\leq C_0$, then with some constant $C$ independent of $N_{\eps}\geq0$, we have $\mathbb{P}(\tilde{X}_{\eps,t}\geq V_F) \geq e^{-C/t}$ for $0<t<\frac{1}{2}$.
\end{lemma} 

\begin{proof}

First, by Chebyshev's inequality we have for $M>0$
\begin{equation}
    \mathbb{P}(\tilde{X}_{\eps,0}\leq -M) \leq M^{-2}\mathbb{E}X_{\eps,0}^2=M^{-2}\mathbb{E}X_{0}^2.
\end{equation} Therefore we can take $M$ large enough such that (for later purpose we also take $M>100\sqrt{a}$)
\begin{equation}
    \mathbb{P}(\tilde{X}_{\eps,0}> -M) =1- \mathbb{P}(\tilde{X}_{\eps,0}\leq -M)\geq 1-M^{-2}\mathbb{E}X_{0}^2>\frac{1}{2}.
\end{equation} 

Treating $bN_{\eps}(t)\geq0$ as external input, the SDE \eqref{SDE-tildeX} can by solved as
\begin{equation}\label{sde-formula}
    \tilde{X}_{\eps,t}=e^{-t}\tilde{X}_{\eps,0}+\int_0^{t}be^{s-t}N_{\eps}(s)ds+\sqrt{2a}\int_0^{t}e^{s-t}dB_s.
\end{equation}
When $\tilde{X}_{\eps,0}\geq -M$, we derive using $bN_{\eps}(t)\geq0$
\begin{align}
    \tilde{X}_{\eps,t}\geq -e^{-t}M+\sqrt{2a}\int_0^{t}e^{s-t}dB_s\geq -M+\sqrt{2a}\int_0^{t}e^{s-t}dB_s.
\end{align}
Therefore we estimate the probability
\begin{align} 
\mathbb{P}(\tilde{X}_{\eps,t}\geq V_F) &\geq \mathbb{P}(\tilde{X}_{\eps,t}\geq V_F, \tilde{X}_{\eps,0}\geq-M)\\&\geq \mathbb{P}(-M+\sqrt{2a}\int_0^{t}e^{s-t}dB_s\geq V_F, \tilde{X}_{\eps,0}\geq-M)\\&=  \mathbb{P}(\sqrt{2a}\int_0^{t}e^{s-t}dB_s\geq V_F+M, \tilde{X}_{\eps,0}\geq-M).
\end{align} As the Brownian motion is independent of initial data, we derive
\begin{align}
    \mathbb{P}(\tilde{X}_{\eps,t}\geq V_F) &\geq \mathbb{P}(\sqrt{2a}\int_0^{t}e^{s-t}dB_s\geq V_F+M, \tilde{X}_{\eps,0}\geq-M)\\&= \mathbb{P}(\sqrt{2a}\int_0^{t}e^{s-t}dB_s\geq V_F+M)\mathbb{P}( \tilde{X}_{\eps,0}\geq-M)\\&\geq \frac{1}{2} \mathbb{P}(\sqrt{2a}\int_0^{t}e^{s-t}dB_s\geq V_F+M),
\end{align} 
where in the last step we use the choice of $M$. Finally, we note that $\sqrt{2a}\int_0^{t}e^{s-t}dB_s$ is a Gaussian random variable with zero mean and variance
\begin{equation}\label{pf-lem-tmp-1}
    \sigma^2(t):={2a}\int_0^{t}e^{2(s-t)}ds\geq \frac{1}{2}at, \quad \text{for} \quad 0<t<\frac{1}{2}.
\end{equation} 
 Therefore, we obtain
\begin{align}\label{pf-lem-tmp-2}
   \mathbb{P}(\tilde{X}_{\eps,t}\geq V_F) \geq \mathbb{P}(\sqrt{2a}\int_0^{t}e^{s-t}dB_s\geq V_F+M)=\Psi(\frac{V_F+M}{\sigma(t)}), 
\end{align} 
where $\Psi(\lambda)$ is the tail probability of the standard Gaussian
\begin{equation}
    \Psi(\lambda):=\int_{\lambda}^{+\infty}\frac{1}{\sqrt{2\pi}}e^{-\frac{1}{2}u^2}du.
\end{equation}
%It has
The reader can check the following (loose) lower bound: $\Psi(\lambda)\geq e^{-2\lambda^2},$ for $\lambda>10$.
%(see e.g. \cite[Prop. 2.1.2]{vershynin2018high}). 
Therefore combining  \eqref{pf-lem-tmp-1} and \eqref{pf-lem-tmp-2} with this tail bound (note that our choice of $M$ ensures $\frac{V_F+M}{\sigma(t)}>10$) we have
\begin{align}
        \mathbb{P}(\tilde{X}_{\eps,t}\geq V_F) \geq e^{-2\frac{(V_F+M)^2}{\sigma^2(t)}}\geq e^{-\frac{4(V_F+M)^2}{at}}=  e^{-\frac{C}{t}}.
\end{align} Then the lemma is proved.

\end{proof}

\bibliographystyle{plain}
\bibliography{Biblio1}

\end{document}